\documentclass[11pt]{amsart}
\usepackage{amssymb}
\usepackage{latexsym}
\usepackage{amsthm}
\usepackage{amscd}
\usepackage{color}

\newtheorem{thm}{Theorem}[section]

\newtheorem{lem}[thm]{Lemma}
\newtheorem{prop}[thm]{Proposition}

\newtheorem{defn}[thm]{Definition}

\numberwithin{equation}{section}

\newcommand{\cj}{\overline}

\newcommand{\al}{\alpha}
\newcommand{\ga}{\gamma}
\newcommand{\Ga}{\Gamma}
\newcommand{\hga}{{\hat{\gamma}}}
\newcommand{\tga}{{\widetilde{\gamma}}}
\newcommand{\ep}{\varepsilon}

\allowdisplaybreaks[1]

\newcommand{\qq}{{q_{\circ}}}
\newcommand{\rr}{{r_{\circ}}}

\begin{document}
\title[Profile decomposition of fractional Schr\"odinger equations]{Profile decompositions \\ of fractional Schr\"odinger equations \\ with angularly regular data}

\author[Y. Cho]{Yonggeun Cho}
\author[G. Hwang]{Gyeongha Hwang}
\author[S. Kwon]{Soonsik Kwon}
\author[S. Lee]{Sanghyuk Lee}

\address{Yonggeun Cho, Department of Mathematics, and Institute of Pure and Applied Mathematics, Chonbuk National University, Jeonju 561-756, Republic of Korea}
\email{changocho@jbnu.ac.kr}
\address{Gyeongha Hwang $\&$ Sanghyuk Lee, Department of Mathematical Sciences, Seoul National University, Seoul 151-747, Republic of Korea}
\email{ghhwang@snu.ac.kr, shklee@snu.ac.kr}
\address{Soonsik Kwon, Department of Mathematical Sciences, Korea Advanced Institute of Science and Technology, Daejeon 305-701, Republic of Korea}
\email{soonsikk@kaist.edu}

\thanks{}

\begin{abstract}
We study the fractional Schr\"odinger equations in $\mathbb R^{1+d},
d \geq 3$ of order ${d}/({d-1}) < \al < 2$. Under the angular
regularity assumption we prove linear and nonlinear profile
decompositions which extend the previous results \cite{chkl2} to
data without radial assumption. As applications we show blowup
phenomena of solutions to mass-critical fractional Hartree
equations.
\end{abstract}

\maketitle
\section{Introduction}
We continue the study of the fractional Schr\"odinger equations with
Hartree type nonlinearity which was carried out in our previous work
\cite{chkl2} under the assumption that the initial data are radial.
Let us consider the following equation with initial data of angular
regularity:
\begin{align}\label{eqn}
\left\{\begin{array}{l} iu_t + (-\Delta)^\frac\alpha2 u = \lambda (|x|^{-\al} * |u|^2)u, \;\;(t, x) \in \mathbb{R}^{1+d},\,\, d \ge 3,\\
u(0, x) = f(x) \in L^2_\rho H^\ga_\sigma,
\end{array}\right.
\end{align}
where $\lambda=\pm1$, $(-\Delta)^\frac\alpha2=\mathcal
F^{-1}|\xi|^\alpha \mathcal F$. Here, the Sobolev space $ L^2_\rho
H^\ga_\sigma$ is defined by the norm \[\|f\|_{L_\rho^2
H_\sigma^\ga}^2 = \int_0^\infty \int_{S^{n-1}}|D_\sigma^\ga
f(\rho\sigma)|^2\,d\sigma \rho^{n-1}\,d\rho\] and the operator
$D_\sigma^\ga$ is given by $(1 - \Delta_\sigma)^\frac{\ga}{2}$ while
$\Delta_\sigma$ is the Laplace-Beltrami operator defined on the unit
sphere. In dimension 3, $\Delta_\sigma$ is the square of angular momentum operator. So, the norm $\|f\|_{L^2_\rho H_\sigma^\ga}$ can be referred as a quantity associated with mass and initial angular momentum. The index $\alpha \in (0, 2)$ is  the fractional order of
equation which is known for L\'{e}vy stability index. In \cite{la1}
Laskin introduced the fractional quantum mechanics in which he
generalized the Brownian-like quantum mechanical path, in the
Feynman path integral approach to quantum mechanics, to the
$\alpha$-stable L\'evy-like quantum mechanical path. The equation
\eqref{eqn} of other types of nonlinearities also appears in
astrophysics or water waves, particularly  with $\alpha = \frac12$
or $\frac32$. See \cite{frohlenz2, iopu} and references therein.

The solutions to equation \eqref{eqn} have the conservation laws for
the mass and the energy:
\begin{align*}
M(u) &= \int |u|^2 \,dx,\qquad E(u) = \frac 12 \int
\cj{u}|\nabla|^\alpha u \,dx - \frac\lambda4 \int \cj{u}(|x|^{-\al} * |u|^2)u\,dx.
\end{align*}
We say  that \eqref{eqn} is focusing if $\lambda =1$, and defocusing
if $\lambda=-1$. The equation \eqref{eqn} is mass-critical, as
$M(u)$ is invariant under scaling $  u (t,x) \to u_\rho (t,x)=
\rho^{-d/2}u({t}/{\rho^{\alpha}},$ ${x}/\rho )$, $\rho>0$, and
$u_\rho$ is again a solution to \eqref{eqn} with initial datum
$\rho^{-d/2}u(0, {x}/\rho)$.  The Cauchy problem \eqref{eqn} is
locally well-posed in $L^2_\rho H^\ga_\sigma$ if $\gamma \ge
\gamma_0$ for a certain $\gamma_0$. See Appendix A. There are
well-posedness results with initial data in different Sobolev
spaces. See \cite{chho, iopu} for results with critical or noncritical
nonlinearity.

In the previous works \cite{ chkl,chkl2} it was intended to extend
the theory of critical nonlinear Schr\"odinger equations to the
fractional order equations. For the focusing case, the authors
\cite{chkl} used a virial argument to show the finite time blowup
with radial data provided that the energy $E(u)$ is negative.  In
\cite{chkl2} the linear profile decomposition for the radial $L^2$
data was established. In this paper, we  extend the profile
decomposition to general data while assuming an extra angular
regularity and apply it to show blowup phenomena of solutions to
mass-critical fractional Hartree equations.


Related to nonlinear  dispersive equations with the critical
nonlinearity, the profile decompositions have been intensively
studied and led to various recent developments. For example, see
\cite{keme}. Profile decompositions for the Schr\"odinger equations
with $L^2$ data were obtained by Merle and Vega \cite{meve} when
$d=2$, Carles and Keraani \cite{ck}, $d=1$, and B\'egout and Vargas \cite{beva}, $d\ge 3$. (Also see \cite{bage, bul, ram} for results on the
wave equation and \cite{sh, ksv} on general dispersive
equations.) These results are based on refinements of Strichartz
estimates (see \cite{mvv, bo1}). There is a different approach which
makes use of the Sobolev imbedding \cite{favi} but such approach is not applicable
especially when the equation is $L^2$-critical.

Our approach here also relies on  a refinement of Strichartz
estimate which strengthens the usual estimate. However, when $\alpha
<2$, due to insufficient dispersion, we do not have a proper linear
estimate for general data which matches with the natural scaling
$u\to u_\rho$. In order to get around this one may consider the
Strichartz estimates which  are accompanied  by a loss of derivative
or integrability. For instance, when one assumes the data is radial
or more regular in angular direction, Strichartz estimates have
wider admissible range. (See \cite{guwa} and \cite{cholee}.) More
precisely we make use of the estimate \eqref{basic}. Thanks to the
extended admissible range of \eqref{basic} it is relatively simpler
to obtain the refinement (see Proposition \ref{ref-str-prop} which
is used for the proof of profile decomposition) but the estimate
\eqref{basic} suffers from large loss of angular regularity which
makes it difficult to use \eqref{basic} directly. Hence we need
smoothing in angular variables to compensate the loss. This is done
in Lemma \ref{44-str} by making use of a bilinear $L^2$ estimate.

 We now denote by $U
(t)f$ the solution of the linear equation $iu_t +
(-\Delta)^\frac\alpha2 u = 0$ with initial datum $f$. Then it is
formally given  by
\begin{align*}
U (t)f = e^{it(-\Delta)^{\frac\alpha2}}f:=\frac1{(2\pi)^d}
\int_{\mathbb{R}^d} e^{i(x\cdot \xi + t
|\xi|^\alpha)}\widehat{f}(\xi)\,d\xi.
\end{align*}
Here $\widehat{f}$ denotes the Fourier transform of $f$ such that
$\widehat{f}(\xi) = \int_{\mathbb{R}^d} e^{-ix\cdot \xi} f(x)\,dx$.

The following is our main result.
\begin{thm}\label{main}
Let  $d \ge 3$, $\frac{d}{d-1} < \alpha < 2$, and $2 < q, r <
\infty$ satisfy $\frac \al q + \frac dr = \frac d2$.  Suppose that
$(u_n)_{n \geq 1}$ is a sequence of complex-valued functions
satisfying $\|u_n\|_{L^2_\rho H^\ga_\sigma} \leq 1$ for some $\ga \ge 0$. Then up to a subsequence, for
any $l \geq 1$, there exist a sequence of functions
$(\phi^j)_{1\leq j \leq l} \subset L^2_\rho H^\ga_\sigma$, $\omega_n^l \in L^2_\rho H^\ga_\sigma$ and a
family of parameters $(h_n^j, t_n^j)_{1 \leq j \leq l, n \geq 1}$
such that
$$u_n(x) = \sum_{1 \leq j \leq l} U(t^j_n)[(h^j_n)^{-d/2}\phi^j(\cdot/{h^j_n})](x) + \omega^l_n(x)$$
and the following properties are satisfied:\\
$(1)$ If   $\hga < \frac {d-1}2 - \frac 1q - \frac {d-1}r$,
then
\begin{align}\label{rem}\lim_{l \rightarrow \infty} \limsup_{n \rightarrow \infty}
\|U_\alpha(\cdot) \omega^l_n\|_{L^q_tL^r_\rho H^{\hga + \ga}_\sigma
(\mathbb{R} \times \mathbb{R}^d)} = 0.\end{align}
$(2)$  For $j \neq k$, $(h^j_n, t^j_n)_{n \geq 1}$ and $(h^k_n,
t^k_n)_{n \geq 1}$ are asymptotically orthogonal in the sense that
\begin{align*}
 \limsup_{n \rightarrow
\infty}\left(\frac{h^j_n}{h^k_n} + \frac{h^k_n}{h^j_n} +
\frac{|t^j_n - t^k_n|}{(h^j_n)^\alpha} + \frac{|t^j_n -
t^k_n|}{(h^k_n)^\alpha} \right)= \infty.
\end{align*}
$(3)$ For each $l\ge 0$,
$$\lim_{n \rightarrow \infty} \Big[\|u_n\|_{L^2_\rho H_\sigma^\ga}^2 - (\sum_{1 \leq j \leq l}\|\phi^j\|_{L^2_\rho H_\sigma^\ga}^2
+ \|\omega^l_n\|_{L^2_\rho H_\sigma^\ga}^2) \Big] = 0.$$
\end{thm}

It should be noted that $\frac {d-1}2 - \frac 1q - \frac {d-1}r$ is positive. So we can take a positive $\hga$ for \eqref{rem}, which is important for proof of nonlinear profile decomposition. The regularity requirement  $\hga < \frac {d-1}2 - \frac 1q - \frac {d-1}r$ is far from being optimal.
Concerning the parameters which appear in Theorem \ref{main}, one notices that the space
translation and modulation(=frequency translation) are absent. It is
not surprising in that they are not noncompact symmetries of the
linear estimate \eqref{basic}. For instance, testing with a
translating sequence $ f_n(\cdot)=f(\cdot - nx_0)$ for $ x_0 \neq 0 $, one can observe
that $\| f_n \|_{L^2} = c $ but $ \| D^{\hat{\gamma}}_\sigma
U(\cdot)f_n \|_{L^q_tL^r_\rho L_\sigma^2} \to 0$ \footnote{In view
of Sobolev embedding, the derivative $D^{\hat{\gamma}}_\sigma$ is not
sufficient to recover the loss of integrability in $L^2_\sigma$,
when compared the usual Strichartz estimate.}.

\smallskip


Once we obtain the linear profile decomposition, we can apply it to
the nonlinear problem \eqref{eqn}. The procedure is now well
established and rather standard. Especially, the equation
\eqref{eqn} in the angularly regular case is similar to the
radial case \cite{chkl2} (For nonlinear Schr\"odinger equation, see \cite{ker2}). Hence, we mostly omit its proof. But for
the sake of  completeness we provide statements of results regarding
the blowup problem (see Section 4).

Using a perturbation argument and global well-posedness for small
data (see Appendix \ref{s-gwp}), we can extend the linear profile decomposition
to the nonlinear profile decomposition. As in the linear profile
decomposition, parameters in the nonlinear profile decomposition
also have the asymptotic orthogonality. Then it is easy to show the
existence of blowup solutions of minimal quantity associated with angular regularity, and the mass
concentration phenomena of finite time blowup solutions. See Section
4 for detail.

\smallskip

The rest of the paper is organized as follows: In Section 2, we will
show the refined Strichartz estimate. Section 3 will be devoted to proving the main theorem, establishing the linear profile
decomposition. In Section 4, we discuss applications to nonlinear
profile decomposition and blowup profile. In Appendix A, we provide
the proof of small data global well-posedness for the Cauchy problem
\eqref{eqn}, as the initial data is not usual Sobolev data. The
proofs of theorems in Section 4 rely on Propositions \ref{CP1},
\ref{wpasy}.

\section{Refined Strichartz Estimates}\label{sec2}
For the fractional Schr\"odinger equation with $\al <2$, it is known that the Strichartz estimate for $L^2$-data has a loss of regularity. However, if one imposes an angular regularity on data, one can recover some of loss of regularity.
Recently, almost optimal range of admissible pairs was established
in \cite{guwa} and the range was further extended in \cite{cholee,
ke} to include the remaining endpoint cases.

For $\frac{d}{d-1}<\alpha < 2$ and $2\le q,r\le \infty$,  let us
set
\[\beta(\al,q,r) = d/2 - d/r - \al /q.\] We now recall from \cite{chl} the estimate
\[\||\nabla|^{-\beta(\al,q,r)}D_\sigma^{\frac {d-2}{2} - \frac
{d-1}{r}}U(\cdot)f \|_{L^2_tL^r_\rho L^2_\sigma} \lesssim
\|f\|_{L^2}\] for $\frac {2(d-2)}{d-2} < r < \infty$. Then by
interpolating with mass conservation (the case $q=\infty$, $r=2$) we
have
\begin{equation}\label{basic}
\||\nabla|^{-\beta(\al,q,r)} D_\sigma^{\hga}U(\cdot)f
\|_{L^q_tL^r_\rho L^2_\sigma} \lesssim \|f\|_{L^2} \end{equation}
holds for $\hga < \frac {d-1}2 - \frac 1q - \frac {d-1}r$ whenever
$q,r\ge 2,\  r \neq \infty$ and $\frac{d}{2}(\frac12 - \frac1r) \le \frac 1q <
(d-1)(\frac 12 - \frac 1r)$.  Using frequency decomposition, we
rewrite \eqref{basic}. Let $P_k, k \in \mathbb Z,$ denote the
Littlewood-Paley operator with symbol $\chi(\xi/2^k) \in C_0^\infty$ which is radial and
supported in the annulus $A_k = \{2^{k-1} < |\xi| \le 2^{k+1}\}$
such that $\sum_{k\in \mathbb Z} P_k=id$.

\begin{lem}\label{str-radial}
Let $\frac{d}{d-1} < \al < 2$, $q,r \geq 2,$ and $r \neq \infty$,
and let $\beta(\al,q,r) = d/2 - d/r - \al /q$.
 If $\frac{d}{2}(\frac12 - \frac1r) \le \frac 1q < (d-1)(\frac 12 - \frac 1r)$,
 then for $\hga < \frac {d-1}2 - \frac 1q - \frac {d-1}r$
$$\|D_\sigma^{\hga}U(\cdot) f \|_{L^q_tL^r_\rho L^2_\sigma} \lesssim \Big(\sum_{k \in \mathbb Z} 2^{2k\beta(\al, q, r)}\|P_kf\|^2_{L^2}\Big)^{1/2}.
$$\end{lem}

When initial data is localized in frequency, it is possible to improve
angular regularity to wider range(for example, see
proof of Lemma 4.1 \cite{cholee} and \cite{glnw}). However, it is not clear that
such estimates can be used to get estimates without frequency
localization.

If we consider interaction of two linear waves of different
frequencies, then we can obtain an improved form of bilinear
Strichartz estimate, which enjoys extra smoothing.
\begin{lem}\label{bi-str}
Let $l \in \mathbb{N}$. Then for $\hga < (d-2)/4$ there exists an
$\epsilon>0$ such that
\[
\Big\|\| D_\sigma^\hga [U(\cdot)P_{\ell}f]\|_{L^2_\sigma}
\|D_\sigma^\hga
[U(\cdot)P_0g]\|_{L^2_\sigma}\Big\|_{L^2(dt|x|^{d-1}d|x|)} \lesssim
2^{\pm\ell(\beta(\alpha, 4, 4) \mp \ep)}\|f\|_2\|g\|_2,
\]
\end{lem}

\begin{proof} Let  $(Y_n^m)$ be the orthonormal spherical harmonic functions
of order $n$. Using the spherical harmonic expansion of $\widehat f$
and $\widehat g$, $ \widehat f(\rho\sigma) = \sum_{n,
m}a_n^m(\rho)Y_n^m(\sigma),$ $ \widehat g(\rho\sigma) = \sum_{n',
m'}b_{n'}^{m'}(\rho)Y_{n'}^{m'}(\sigma),$
we rewrite $U(t)P_\ell f$ and $U(t)P_0 g$ as
\begin{align*}
&U(t)P_\ell f(x) = r^{-\frac{d-2}2}\sum_{n,m}c_n\mathcal T_n^\ell(a_n^m)(t, r)\, Y_n^m(\sigma),\\
&U(t)P_0 g(x) = r^{-\frac{d-2}2}\sum_{n',m'}c_{n'}\mathcal T_{n'}^0(b_{n'}^{m'})(t, r) \,Y_{n'}^{m'}(\sigma),
\end{align*}
where \[\mathcal T_n^\ell(a_n^m)(t, r) = \int_0^\infty
e^{it\rho^\alpha} J_{\nu(n)}(r\rho) \psi(\rho/2^\ell) \rho^\frac d2
a_n^m(\rho)\,d\rho,\] and  $x = r\sigma$, $|c_n| = c$ independent of
on $n, m$, and $\nu(n) = n + (d-2)/2$. See \cite{sw} for detail.

By the spectral theory we have
\begin{equation}\label{spectral}
\langle D_\sigma^\gamma Y_n^m, D_\sigma^\gamma
Y_{n'}^{m'}\rangle_\sigma = (1 + n(n+d-2))^\gamma \delta_{n,
n'}\delta_{m, m'},
\end{equation}
where $\langle u, v \rangle_\sigma = \int_{S^{d-1}}
u\overline{v}\,d\sigma$ and $\delta_{n,n'}$ is the Kronecker delta.
By  this it follows that for $j=0, \ell$
\begin{align*}
\|U(t)P_j f\|_{L^2_\sigma}^2 &= c^2r^{-(d-2)} \sum_{n,m} |\mathcal T_n^j(a_n^m)(t, r)|^2,\\
\|D^{\frac {d-2}2}_\sigma U(t)P_j f\|_{L^2_\sigma}^2 &=
c^2r^{-(d-2)} \sum_{n, m} (1 + n(n + d -2))^{\frac {d-2}2} |\mathcal
T_n^j(a_n^m)(t,r)|^2.
\end{align*}
And by the change of variables $\rho \mapsto \rho^{\frac 1\alpha}$,
we obtain
\begin{align*}
\|D_\sigma^{\frac {d-2}2} &U(t)P_\ell f\|_{L^2_\sigma}^2
\|U(t)P_0g\|_{L^2_\sigma}^2= c^4r^{-2(d-2)}
\sum_{n,n',l,l'}(1 + n(n + d - 2))^{\frac {d-2}2}|\mathcal T_n^\ell(a_n^m)\overline{\mathcal T_{n'}^0(b_{n'}^{m'})}|^2\\
&\lesssim r^{-2(d-2)}\sum_{n,n',m,m'} (1 + n(n + d -2))^{\frac {d-2}2} \left|\int^\infty_0\int^\infty_0 e^{it(\rho - \rho')} F^m_n(\rho) G^{m'}_{n'}(\rho')\, d\rho d\rho'\right|^2\\
&\lesssim  r^{-2(d-2)}\sum_{n,n',m,m'} (1 + n(n + d -2))^{\frac
{d-2}2}|\widehat{\widetilde{F^m_n}*\overline{G^{m'}_{n'}}}(t)|^2,
\end{align*}
where $\widetilde{F_n^m}(\rho) = F_n^m(-\rho)$, and
\begin{align*}
&F_n^m(\rho) = \chi_{(0, \infty)}(\rho)J_{\nu(n)}(r\rho^\frac1\alpha)(\rho)^\frac{d-2(\alpha-1)}{2\alpha} \psi(\rho^\frac1\alpha/2^\ell)a_n^m(\rho^\frac1{\alpha}),\\
&G_n^m(\rho') = \chi_{(0, \infty)}(\rho')J_{\nu(n)}(r{\rho'}^\frac1\alpha) (\rho')^\frac{d-2(\alpha-1)}{2\alpha} \psi({\rho'}^\frac1\alpha)b_{n}^{m}({\rho'}^\frac1\alpha).
\end{align*}
Taking $L^2$-norm with respect to $t$-variable, from Plancherel's theorem and Young's convolution inequality, we get
\begin{align}\begin{aligned}\label{sigma-t int}
&\Big\| \|D^{\frac {d-2}2}_\sigma U(t) P_\ell f\|_{L^2_\sigma} \|U(t)P_0g\|_{L^2_\sigma} \Big\|_{L^2(dt)}\\
&\qquad \lesssim  r^{-(d-2)}\Big(\sum_{n,n',m,m'} (1 + n(n + d
-2))^{\frac {d-2}2}
\|F^m_n\|_{L^2(d\rho)}^2\|G^{m'}_{n'}\|_{L^1(d\rho')}^2\Big)^\frac
12.
\end{aligned}
\end{align}
We then take $L^2(r^{d-1}dr)$ on both sides of \eqref{sigma-t int} to obtain
\begin{align}\begin{aligned}\label{sigma-tr int}
&\Big\| \|D^{\frac {d-2}2}_\sigma U(t) P_\ell f\|_{L^2_\sigma} \|U(t)P_0g\|_{L^2_\sigma} \Big\|_{L^2(r^{d-1}drdt)}^2\\
&\qquad\lesssim (1+n)^{d-2} \sum_{n,n',m,m'} \int^\infty_0 r^{-(d-2)} \|F^m_n\|_{L^2(d\rho)}^2\,dr \cdot \sup_{r}(r\|G^{m'}_{n'}\|_{L^1(d\rho')}^2)\\
&\qquad:= \sum_{n,n',m,m'} (1+n)^{d-2} [A^m_n]^2 [B^{m'}_{n'}]^2,
\end{aligned}\end{align}
where \[A^m_n = \Big(\int^\infty_0 r^{-(d-2)}
\|F^m_n\|_{L^2(d\rho)}^2\,dr\Big)^{\frac 12}, \quad B^{m'}_{n'} =
\Big(\sup_{r}(r\|G^{m'}_{n'}\|_{L^1(d\rho')}^2)\Big)^{\frac 12}.\]
Making the change of variables $\rho \mapsto \rho^\alpha$, we have
for $A_n^m$ that
$$
[A_n^m]^2 = \int_0^\infty \left(\int_0^\infty  |J_{\nu(n)}(r\rho)|^2r^{-(d-2)}\,dr\right) \rho^{d-(\alpha-1)}(\psi(\rho/2^\ell))^2|a_n^m(\rho)|^2\,d\rho.
$$
From the Bessel function estimate (see p. 403 of \cite{wat}) and the Stirling's formula it follows that
\begin{align*}
\int_0^\infty  |J_{\nu(n)}(r\rho)|^2r^{-(d-2)}\,dr &= \frac{(\rho/2)^{d-3}\Gamma(d-2)\Gamma(\nu(n) - \frac{d-3}2)}{2(\Gamma(\frac{d-1}2))^2\Gamma(\nu(n) + \frac{d-1}2)} \lesssim \rho^{d-3}(1 + n)^{-(d-2)}.
\end{align*}
Thus we have
\begin{align}\label{akl}
[A_n^m]^2 \lesssim (1 + n)^{-(d-2)}2^{\ell(d- \alpha-1)}\int |a_n^m(\rho)|^2\rho^{d-1}\,d\rho.
\end{align}
For $B_{n'}^{m'}$, after change of variables, we estimate
\begin{align*}
[B_{n'}^{m'}]^2 &\lesssim \sup_r r\left( \int |J_{\nu(n')}(r\rho')| {\rho'}^\frac d2 \psi(\rho')|b_{n'}^{m'}(\rho')|\,d\rho'\right)^2\\
&\lesssim \sup_r \int_{\rho' \sim 1} |J_{\nu(n')}(r\rho')|^2r\rho'\,d\rho'\int |b_{n'}^{m'}(\rho')|^2 {\rho'}^{d-1}\,d\rho'\\
&\lesssim \sup_r \int_{\rho'\sim r}|J_\nu(n')(\rho')|^2\,d\rho'\int |b_{n'}^{m'}(\rho')|^2 {\rho'}^{d-1}\,d\rho'.
\end{align*}
Since $\sup_r (\int_{\rho'\sim r}|J_\nu(n')(\rho')|^2\,d\rho') \lesssim 1$ (see (3.4) of \cite{cholee}),
we have
\begin{align}\label{bkl}
[B_{n'}^{m'}]^2 \lesssim \int |b_{n'}^{m'}(\rho')|^2 {\rho'}^{d-1}\,d\rho'.
\end{align}
Plugging \eqref{akl}, \eqref{bkl} into \eqref{sigma-tr int}, we
obtain
\begin{align}\begin{aligned}\label{aklbkl}
&\Big\| \|D^{\frac {d-2}2}_\sigma U(t) P_\ell f\|_{L^2_\sigma} \|U(t)P_0g\|_{L^2_\sigma} \Big\|_{L^2(r^{d-1}drdt)}\\
&\qquad \lesssim 2^{\ell(d-\alpha-1)/2}\left(\sum_{n,m}\int |a_n^m|^2\rho^{d-1}\,d\rho\right)^\frac12\left(\sum_{n',m'}\int |b_{n'}^{m'}|^2{\rho'}^{d-1}\,d\rho'\right)^\frac12\\
&\qquad \lesssim 2^{\ell(d-\alpha-1)/2} \|f\|_2\|g\|_2.
\end{aligned}\end{align}

On the other hand, repeating the above argument we have
\begin{align*}
&\|U(t)P_\ell f\|_{L^2_\sigma}^2 \|D^{\frac {d-2}2}_\sigma U(t)P_0g\|_{L^2_\sigma}^2\\
&\qquad \lesssim  r^{-(d-2)}\sum_{n, n', m, m'} (1 + n'(n' + d
-2))^{\frac {d-2}2}
|\widehat{\widetilde{F^m_n}*\overline{G^{m'}_{n'}}}(t)|^2.
\end{align*}
Hence by using Young's convolution inequality again as in \eqref{sigma-t int}, we get
\begin{align}\begin{aligned}\label{sigma-t int2}
&\Big\| \|U(t) P_\ell f\|_{L^2_\sigma} \|D^{\frac {d-2}2}_\sigma U(t)P_0g\|_{L^2_\sigma} \Big\|_{L^2(dt)}\\
&\qquad\lesssim  r^{-(d-2)}\Big(\sum_{n, n',m,m'} (1 + n'(n' + d
-2))^{\frac {d-2}2}
\|F^m_n\|_{L^1(d\rho)}^2\|G^{m'}_{n'}\|_{L^2(d\rho')}^2\Big)^\frac
12.
\end{aligned}\end{align}
And we also have
\begin{align*}
&\Big\| \|U(t) P_\ell f\|_{L^2_\sigma} \|D^{\frac {d-2}2}_\sigma U(t)P_0g\|_{L^2_\sigma} \Big\|_{L^2(r^{d-1}drdt)}^2
\lesssim \sum_{n,n',m,m'} (1+n)^{d-2} [\widetilde{A}^m_n]^2
[\widetilde{B}^{m'}_{n'}]^2,
\end{align*}
where $\widetilde{A}^m_n =
(\sup_{r}(r\|F^m_n\|_{L^1(d\rho')}^2))^{\frac 12}$ and
$\widetilde{B}^{m'}_{n'} = (\int^\infty_0 r^{-(d-2)}
\|G^{m'}_{n'}\|_{L^2(d\rho)}^2\,dr)^{\frac 12}$. Changing the role
of $F_n^m$ and $G^{m'}_{n'}$ in \eqref{akl} and \eqref{bkl}, one can
easily show that
\begin{align*}
[\widetilde{A}_n^m]^2 \lesssim 2^{\ell}\int
|a_n^m(\rho)|^2\rho^{d-1}\,d\rho, \, \, [\widetilde{B}_{n'}^{m'}]^2
\lesssim (1 + n')^{-(d-2)}\int |b_{n'}^{m'}(\rho')|^2
{\rho'}^{d-1}\,d\rho',
\end{align*}
which implies
\begin{align}\begin{aligned}\label{bklakl}
&\Big\| \|U(t) P_\ell f\|_{L^2_\sigma} \|D^{\frac {d-2}2}_\sigma
U(t)P_0g\|_{L^2_\sigma} \Big\|_{L^2(r^{d-1}drdt)} \lesssim
2^{\ell/2} \|f\|_2\|g\|_2.
\end{aligned}\end{align}
Finally, interpolation between \eqref{aklbkl} and \eqref{bklakl}
gives the desired estimate.
\end{proof}

In \cite{chkl2} a refined Strichartz estimate is shown for radial
function. We extend it to the functions with angular regularity.
This extension will play a crucial role in proving profile
decomposition of angularly regular data. Here we combine the
argument in \cite{chkl2, chl} with the spherical harmonic expansion.

For $\alpha<2$, we say the pair $(q,r)$ $\al-$admissible, if $\frac
\al q + \frac dr = \frac d2$ for $2 \le q, r \le \infty$.
\begin{prop}\label{ref-str-prop}
Let $\frac{d}{d-1} < \al < 2$, $q > 2$ and $r \neq \infty$. Then for each $\alpha$-admissible pair
$(q,r)$ there exist $\theta, p$ with $\theta
\in (0,1), p \in [1, 2)$ such that for any $\hga < \frac {d-1}2 - \frac 1q - \frac {d-1}r$,
\begin{align}\label{ref-str}
\|D_\sigma^{\hga}U(\cdot)f\|_{L^q_tL^r_\rho L^2_\sigma} \lesssim \big(\sup_k 2^{kd(\frac 12 -
\frac 1p)} \|\widehat{P_kf}\|_p\big)^\theta \|f\|_2^{1-\theta}.
\end{align}
\end{prop}
In order to show Proposition \ref{ref-str-prop}, we need the
following lemma.
\begin{lem}\label{44-str} Let $d\ge 3$. Then for $\hga < (d-2)/4$
\begin{align}\label{l4}\|D_\sigma^{\hga}U(\cdot)f\|_{L^4_tL^4_\rho L^2_\sigma} \lesssim \Big(\sum_k(2^{k\beta(\al, 4,4)}
\|\widehat{P_kf}\|_2)^4\Big)^{\frac 14}.\end{align}
\end{lem}

\begin{proof}
By Littlewood-Paley decomposition we write
\begin{align*}
\langle D_\sigma^\hga U(t)f, D_\sigma^\hga U(t)f \rangle_\sigma =
\sum_{j =-\infty}^\infty \sum_k \langle D_\sigma^\hga U(t)P_kf,
D_\sigma^\hga U(t)P_{j+k}f \rangle_\sigma.
\end{align*}
For \eqref{l4} it is sufficient to show that  for some $\epsilon>0$
\begin{align}\label{l2-half}
 \|\sum_k \langle D_\sigma^\hga U(\cdot) P_kf, D_\sigma^\hga U(\cdot) P_{j+k}f \rangle_\sigma \|_{L^2_tL^2_\rho}
\lesssim
2^{-|j|\epsilon} \Big(\sum_k \big(2^{k\beta(\al,
4,4)}\|\widehat{P_kf}\|_2\big)^4\Big)^{1/2}.
\end{align}

We show the cases $|j|\le 3$ and  $|j| > 3$, separately. Let us
first consider the case $|j|\le 3$. By the H\"older's and the Cauchy-Schwarz
inequalities, we have
\begin{align*}
&\qquad\qquad \Big|\sum_k \langle D_\sigma^\hga U(t) P_kf,
\,D_\sigma^\hga U(t) P_{j+k}f \rangle_\sigma \Big|^2
\lesssim \Big(\sum_k \|D_\sigma^\hga U(t) P_kf \,D_\sigma^\hga U(t) P_{j+k}f\|_{L^1_\sigma}\Big)^2\\
& \lesssim \Big(\sum_k \|D_\sigma^\hga U(t)
P_kf\|_{L^2_\sigma}\|D_\sigma^\hga U(t)
P_{j+k}f\|_{L^2_\sigma}\Big)^2\lesssim \sum_{l=-\infty}^\infty
\sum_k \Big(\|D_\sigma^\hga U(t)P_kf\|_{L^2_\sigma}\|D_\sigma^\hga
U(t)P_{k+l}f\|_{L^2_\sigma}\Big)^2.
\end{align*}
Hence we get
\begin{align*}
\mbox{LHS of}\;\eqref{l2-half} \lesssim \sum_{l=-\infty}^\infty
\sum_k \Big\| \|D_\sigma^\hga U(t)P_k
f\|_{L^2_\sigma}\|D_\sigma^\hga
U(t)P_{k+l}f\|_{L^2_\sigma}\Big\|_{L_t^2L_\rho^2}^2.
\end{align*}
Lemma \ref{bi-str} and the Cauchy-Schwarz inequality imply
\eqref{l2-half} when $|j|\le 3$.

We now consider the case $|j| > 3$. As previously, using the
spherical harmonic expansion such that $\widehat f = \sum a_n^m
Y_n^m$ and $\widehat g = \sum b_{n'}^{m'} Y_{n'}^{m'}$  and making the change of
variables $\rho \mapsto \rho^\alpha$,  we get
\begin{align*}
&\langle D_\sigma^\hga [U(t)P_k f], D_\sigma^\hga[U(t) P_{j+k}g]\rangle_\sigma\\
 &\qquad = c^2\alpha^{-2}r^{-(d-2)}\sum_{n, m} (1 + n(n+d-2))^\frac{d-\hga}4 \int_\mathbb R\!\!\int_\mathbb R e^{it(\rho' - {\rho''})}F_n^m(\rho')\overline{G_n^m(\rho'')}\,d\rho' d\rho''\\
 &\qquad = c^2\alpha^{-2}r^{-(d-2)}\sum_{n, m} (1 + n(n+d-2))^\frac{d-\hga}4 (\widehat{\widetilde{F_{n}^{m}} * \overline{G_{n}^{m}}})(t),
\end{align*}
where
\begin{align*}
&F_{n}^{m}(\rho) = \chi_{(0, \infty)}(\rho)J_{\nu(n)}(r{\rho}^\frac1\alpha)\rho^\frac{d-2(\alpha-1)}{2\alpha} \psi({\rho}^\frac1\alpha/2^k)a_{n}^{m}({\rho}^\frac1{\alpha}),\\
&G_{n}^{m}(\rho') = \chi_{(0, \infty)}(\rho')J_{\nu(n)}(r{\rho'}^\frac1\alpha) (\rho')^\frac{d-2(\alpha-1)}{2\alpha} \psi({\rho'}^\frac1\alpha/2^{j+k})b_{n}^{m}({\rho'}^\frac1\alpha).
\end{align*}
Hence the Fourier support of $\langle D_\sigma^\hga
[U(t)P_k f], D_\sigma^\hga[U(t) P_{j+k}g]\rangle_\sigma$
with respect to $t$ is $\{\tau : 2^{\alpha(j+k-2)} \leq \tau \leq
2^{\alpha(j+k+2)}\}$. So the Fourier supports of $\langle
D_\sigma^\hga U(t)P_kf, D_\sigma^\hga U(t)P_{j+k}f\rangle_\sigma$
with respect to $t$ are boundedly overlapping. Then Plancherel's
theorem in $t$ gives
\begin{align*}
\mbox{LHS of}\;\eqref{l2-half} &= \|\sum_k \langle D^\hga_\sigma U(\cdot) P_kf, \, D^\hga_\sigma
U(\cdot)P_{j+k}f \rangle_\sigma \|_{L^2_\rho L^2_t}^2
\\
&\lesssim \sum_k \| \langle D_\sigma^\hga  U(t)P_k f, \,
D_\sigma^\hga  U(t)P_{j+k}f \rangle_\sigma \|_{L_\rho^2L_t^2}^2.
\end{align*}
Hence using Cauchy-Schwarz and Lemma \ref{bi-str} we see that
\begin{align*}
&\mbox{LHS of}\;\eqref{l2-half} \lesssim \sum_k \big\| \|
D_\sigma^\hga U(t)P_k f\|_{L^2_\sigma}
\|D_\sigma^\hga  U(t)P_{j+k}f \|_{L^2_\sigma} \big\|_{L_t^2L_\rho^2}^2\\
&\lesssim \sum_k
2^{-|j|\epsilon}2^{2k\beta(\al,4,4)}\|P_kf\|_2^22^{2(j+k)\beta(\al,4,4)}\|P_{j+k}f\|_2^2
\lesssim 2^{-|j|\epsilon} \sum_k (2^{k\beta(\al,4,4)}\|P_kf\|_2)^4.
\end{align*}
For the last inequality we used the Cauchy-Schwarz inequality. This completes the proof.
\end{proof}
\noindent Now we are ready to prove Proposition \ref{ref-str-prop}.
\begin{proof} To begin with we note that $\beta(\alpha,q,r) = 0$
because $(q,r)$ is $\alpha$-admissible. From Lemma \ref{str-radial}
we have
\begin{align}\label{ref-1}
\|D_\sigma^{\hga}U(\cdot)f\|_{L^q_tL^r_\rho L^2_\sigma} \lesssim
\Big(\sum_k \|\widehat{P_kf}\|_2^2\Big)^{1/2}
\end{align}
for any $\al$-admissible pair $(q, r)$. Then \eqref{ref-str} follows
from interpolation of \eqref{ref-1} and the following two estimates:
for some $p_*,q_*$ with $p_* < 2 < q_*$,
\begin{align}
\|D_\sigma^{\hga}U(\cdot)f\|_{L^q_tL^r_\rho L^2_\sigma} &\lesssim \Big(\sum_k \| \widehat{P_kf} \|_2^{q_*}\Big)^{1/{q_*}},\label{ref-2}\\
\|D_\sigma^{\hga}U(\cdot)f\|_{L^q_tL^r_\rho L^2_\sigma} &\lesssim \sum_k 2^{kd(\frac 12 - \frac 1{p_*}) } \| \widehat{P_kf}
\|_{p_*}.\label{ref-3}
\end{align}
In fact, the interpolation among \eqref{ref-1}, \eqref{ref-2} and
\eqref{ref-3} gives
\begin{equation}\label{intpol}
\|D_\sigma^{\hga}U(\cdot)f\|_{L^q_tL^r_\rho L^2_\sigma} \lesssim \Big(\sum_k\big(2^{kd(\frac 12 -
\frac 1{p_0})} \| \widehat{P_kf} \|_{p_0}\big)^{q_0}\Big)^{1/{q_0}}
\end{equation} for $(1/{q_0},1/{p_0})$ contained in the triangle
with the vertices $(1/2,1/2)$, $(1/{p_*}, 1)$ and $(1/2,
1/{q_*})$. So, there exist $q_0,p_0$, $p_0 < 2 < q_0$, for which
\eqref{intpol} holds. Hence,
\begin{align*}
\|D_\sigma^{\hga}U(\cdot)f\|_{L^q_tL^r_\rho L^2_\sigma} &\lesssim \Big(\Big(\sup_k 2^{kd(\frac 12 -
\frac 1{p_0})} \|\widehat{P_kf}\|_{p_0}\Big)^{q_0-2} \sum_k
\big(2^{kd(\frac 12 - \frac 1{p_0})} \| \widehat{P_kf}
\|_{p_0}\big)^{2}\Big)^{1/{q_0}}
\\
&\lesssim \Big(\sup_k 2^{kd(\frac 12 - \frac 1{p_0})}
\|\widehat{P_kf}\|_{p_0}\Big)^{(q_0-2)/q_0}\Big(
\sum_k \| \widehat{P_kf} \|_{2}^{2}\Big)^{1/{q_0}}\\
&\lesssim \Big(\sup_k 2^{kd(\frac 12 - \frac 1{p_0})}
\|\widehat{P_kf}\|_{p_0}\Big)^{(q_0-2)/q_0} \|f\|_2^{2/q_0}.
\end{align*}
We need only to set $p=p_0$ and $\theta=1-2/q_0$ to get \eqref{ref-str}.

Now we show \eqref{ref-2} and \eqref{ref-3}. First we consider the
inequality \eqref{ref-3}.  Let $(q,r)$  be an $\alpha$-admissible
pair with $2<q<\infty$. Since $\hga < \frac {d-1}2 - \frac {1}{q} - \frac {d-1}{r}$, $\hga$ is less than $\frac {d-1}2 - \frac {1}{q_0} - \frac {d-1}{r_0}$ on a small neighborhood of $q,r$. One can easily check that $\frac d2 - \frac d{6} - \frac
\al{6} > 0$. Hence there exist
$2< q_0,r_0<\infty$ such that $\frac d2 - \frac d{r_0} - \frac
\al{q_0} < 0$,  $\frac1{q_0} \le (d-1)(\frac12 - \frac1{r_0})$,
$\hga < \frac {d-1}2 - \frac {1}{q_0} - \frac {d-1}{r_0}$ and
$(\frac1q,\frac1r)=\theta(\frac1{q_0},\frac1{r_0}) + (1 -
\theta)(\frac 16, \frac 16)$, $0<\theta<1$. For this pair $(q_0,
r_0)$ we have $ \|D_\sigma^{\hga}U(\cdot) P_0f
\|_{L^{{q_0}}_tL^{{r_0}}_\rho L^2_\sigma} \lesssim
\|\widehat{P_0f}\|_2$ from \eqref{basic}. Then interpolating this with
the estimate
\[\|D_\sigma^{\hga}U(\cdot) P_0f\|_{L^6_t L^6_\rho L_\sigma^2}
\lesssim \|D_\sigma^{\frac {d-2^+}{6}}U(\cdot) P_0f\|_{L^6_t
L^6_\rho L_\sigma^2} \lesssim \|\widehat{P_0f}\|_{\frac 32},\] which
follows from $\|D_\sigma^{\frac {d-2^+}{4}}U(\cdot) P_0f\|_{L^4_t
L^4_\rho L_\sigma^2} \lesssim \|\widehat{P_0f}\|_{2}$ and the
trivial estimate $\|U(\cdot) P_0f\|_{L^\infty_t L^\infty_\rho
L^2_\sigma} \lesssim \|U(\cdot) P_0f\|_{L^\infty_t L^\infty_\rho
L^\infty_\sigma}$ $\lesssim \|\widehat{P_0f}\|_1$, implies
$$\|D_\sigma^{\hga}U(\cdot) P_0f\|_{L^q_tL^r_\rho L^2_\sigma}
\lesssim \|\widehat{P_0f}\|_{p_*}$$  for some $p_*$ with $1<p_*<2$.
By rescaling it follows $\|D_\sigma^{\hga}U(\cdot)
P_kf\|_{L^q_tL^r_\rho L^2_\sigma} \lesssim 2^{kd(\frac 12 - \frac
1{p_*})}\|\widehat{P_kf}\|_{p_*}$. Now Minkowski's inequalities
give
\begin{align*}
\|D_\sigma^{\hga}U(\cdot)f\|_{L^q_tL^r_\rho L^2_\sigma} &\le \sum_k \|D^\hga_\sigma U(\cdot)
P_kf\|_{L^q_tL^r_\rho L^2_\sigma} \lesssim \sum_k
2^{kd(\frac 12 - \frac
1{p_*})}\|\widehat{P_kf}\|_{p_*}.
\end{align*}

Finally, the proof of \eqref{ref-2} is easy. It can be done by
interpolating \eqref{l4} with the estimates in Lemma
\ref{str-radial} for $(q, r)$ with $\frac d\al(\frac12-\frac1r) <
\frac1q \le (d-1)\Big(\frac12-\frac1r\Big)$, $2 \le q, r \le
\infty$. This completes the proof.
\end{proof}

\section{Linear profile decomposition}
In this section we prove Theorem \ref{main}. Throughout this section
we assume that $\frac{d}{d-1} < \alpha < 2$, the pair $(q ,r)$ is
$\alpha$-admissible with $\frac{d}2(\frac12 - \frac1r) < \frac1q <
(d-1)(\frac12-\frac1r)$, and $\hga < \frac {d-1}2 - \frac 1q - \frac
{d-1}r$.

\subsection{Preliminary decomposition}
Thanks to the refined Strichartz estimate \eqref{ref-str}, we can
extract frequencies and scaling parameters  to get a preliminary
decomposition.

\begin{prop}\label{prop-decom1}
Let $(u_n)_{n\geq1}$ be a sequence of complex valued functions
satisfying $\|u_n\|_{L^2_\rho H^\gamma_\sigma} \leq 1$ for some
$\gamma \ge 0$. Then for any $\delta > 0$, there exists $N =
N(\delta)$, $\rho^j_n \in (0,\infty)$ and  $(f^j_n)_{1 \leq j \leq
N, n \geq 1} \subset L^2_\rho H^\gamma_\sigma$ such that
$$u_n = \sum^N_{j=1} f_n^j + q^N_n$$
and  the following properties hold:
\begin{enumerate}
\item[$1)$] There exists a compact set $K = K(N)\subset \{\xi : R_1 < |\xi| < R_2\}$
satisfying that
$$(\rho^j_n)^{d/2}|\widehat{D^\ga_\sigma f^j_n}(\rho^j_n \xi)| \leq C_\delta \chi_K(\xi)\;\;\mbox{for every}\;\; 1 \leq j \leq N,$$
\item[$2)$] $\limsup_{n \rightarrow \infty}(\frac{\rho^j_n}{\rho^k_n}+\frac{\rho^k_n}{\rho^j_n}) =
 \infty,\;\;\mbox{ if}\;\; j \neq
k,$
\item[$3)$] $\limsup_{n \rightarrow \infty} \|U(\cdot)q^N_n\|_{L^q_tL^r_\rho H^{\ga + \hga}_\sigma} \leq \delta \;\;\mbox{for any}\;\;N \geq 1,$
\item[$4)$] $\limsup_{n \rightarrow \infty}(\|u_n\|_{L_\rho^2H_\sigma^\ga}^2 - (\sum_{j=1}^N\|f^j_n\|^2_{L_\rho^2H_\sigma^\ga} + \|q^N_n\|^2_{L_\rho^2H_\sigma^\ga})) = 0.$
\end{enumerate}
\end{prop}

\begin{proof} The argument here is similar to the one for the radial
case \cite{chkl2}. To begin with, let us set \[v_n :=
D_\sigma^{\ga}u_n.\] We may assume that $\|U(\cdot ) v_n
\|_{L^q_tL^r_\rho H^{\hga}_\sigma}
> \delta$ for all $n \ge 1$, otherwise there is nothing to prove.
For $\rho_n^1>0$ let us set $A^1_n = \{\xi : \frac{\rho^1_n}2 <
|\xi| < \rho^1_n \}$ and  $\widehat{v_n^1} =
\widehat{v_n}\chi_{A_n^1}.$ By the refined Strichartz estimates
(Proposition \ref{ref-str-prop}), there exists $\rho_n^1$ such that
\begin{align}\label{lb}
c_1(\rho^1_n)^{d(\frac 1p - \frac 12)}\delta^{\frac 1\theta} \leq
\|\widehat{v_n^1}\|_p,
\end{align}
for some positive constant $c_1$, $p \in (1, 2)$, $\theta \in (0,1)$
as stated in Proposition \ref{ref-str-prop}. And since
$\|\widehat{v^1_n}\|_{L^2} \leq \|\widehat{v_n}\|_{L^2} \leq 1$, $
\int_{\{|\widehat{v_n^1}| > \lambda \}} |\widehat{v_n^1}|^p d\xi =
\int_{\{|\widehat{v_n^1}| > \lambda\}}
(\lambda^{2-p}|\widehat{v_n^1}|^p)\lambda^{p-2}d\xi \le
\lambda^{p-2} $  for any $\lambda
> 0$.  Thus we have
$$
\Big(\int_{\{|\widehat{v_n^1}| > \lambda\}} |\widehat{v_n^1}|^p d\xi
\Big)^{\frac 1p} \le \lambda^{1-\frac 2p}.
$$
Now let us set $\lambda = (c_1/ 2^{\frac1p})^{\frac p{2-p}}(\rho^1_n)^{-\frac d2}\delta^{\frac 1\theta \cdot \frac p{p-2}}$. Then \eqref{lb} gives
\begin{align*}
\frac{c_1}2(\rho^1_n)^{d(\frac 1p - \frac 12)}(\min (\rho^1_n,(\rho^1_n)^{-1}))^\ep\delta^{\frac 1\theta}
&\leq \Big(\int_{\{ |\widehat{v^1_n}| < \lambda\}}
|\widehat{v^1_n}|^p\Big)^{\frac 1p} \le |A_n^1|^\frac{2-p}{2p}\Big(\int_{\{|\widehat{v^1_n}| < \lambda\}}
|\widehat{v^1_n}|^2\Big)^{\frac 12}\\
 &\lesssim (\omega_d^\frac1d\rho^1_n)^{d(\frac 1p - \frac12)}\Big(\int_{\{|\widehat{v^1_n}| < \lambda\}}|\widehat{v^1_n}|^2\Big)^{\frac 12},
\end{align*}
where $\omega_d$ is the measure of unit sphere. This implies
$$
\frac {c_1'}2\delta^{\frac 1\theta} \leq
\Big(\int_{\{|\widehat{v^1_n}| < \lambda\}}
|\widehat{v^1_n}|^2\Big)^{\frac 12}\qquad (c_1' =
c_1\omega_d^{1/2-1/p}).
$$

Now define $G_n^1(\psi)(\xi)$ by $(\rho_n^1)^{d/2}
\psi(\rho_n^1\xi)$ for measurable function $\psi$. Then by letting
$\widehat{w^1_n} = \widehat{v^1_n}\; \chi_{\{|\widehat{v^1_n}| <
\lambda\}}$ we get $\|w^1_n\|_2 \geq \frac 12 c_1'\delta^{\frac
1\theta}$ and $|G_n^1(\widehat{w^1_n}(\xi))| = |(\rho^1_n)^{\frac d2}
\widehat{w^1_n}(\rho^1_n \xi)| \leq C_\delta \chi_{A_{1/2, 1}}(\xi)$.
Here $A_{R_1, R_2}$ is the annulus $\{\xi : R_1 < |\xi| < R_2\}$.
We can repeat the above process with $v_n - w^1_n$ replacing $v_n$.
After $N' ( = \widetilde N'(\delta))$ steps, we get
$(w^j_n)_{1 \leq j \leq N'}$ and $(\rho^j_n)$ such that
\begin{align}\label{decomp}
v_n = \sum_{j=1}^{N'} w^j_n + \widetilde{q^{N'}_n}, \;\;\|v_n\|^2_2 = \sum_{j=1}^{N'} \|w^j_n\|^2_2 + \|\widetilde{q^{N'}_n}\|^2_2,\;\;
\|U(t)\widetilde{q^{N'}_n}\|_{L^q_tL^r_\rho H^\hga_\sigma} \leq \delta.
\end{align}
The second identity follows from disjointness of Fourier supports of
$w_n^j$ and $\widetilde{q_n^{N'}}$. In the sequel, we say
$\rho^j_n$ is orthogonal to $\rho^k_n$ if $\limsup_{n\to\infty}
(\frac {\rho^k_n}{\rho^j_n} + \frac {\rho^j_n}{\rho^k_n}) = \infty$.
Define $\widetilde{f^1_n}$ to be the sum of those $w^j_n$ whose
$\rho^j_n$ are not orthogonal to $\rho^1_n$. Take least $j_0 \in
[2,N']$ such that $\rho^{j_0}_n$ is orthogonal to $\rho^1_n$ and
define $\widetilde{f^2_n}$ to be the sum of $w^j_n$ whose $\rho^j_n$
are orthogonal to $\rho_n^1$ but not to $\rho_n^{j_0}$. After $N$
step for $N \le N'$, we have $(\widetilde{f^j_n})_{1 \leq j \leq
N}$. Let us set \[f_n^j = D_\sigma^{-\gamma}\widetilde{f_n^j},\quad
 q_n^{N} = D_\sigma^{-\gamma}\widetilde{q_n^{N'}}.\] Then $\{f_n^j\}$
satisfy the properties 2) and 4) of Proposition \ref{prop-decom1}
because the disjointness of the Fourier supports are disjoint and
the third inequality of \eqref{decomp} gives Property 3).

Now it remains to check Property 1). We only consider $f_n^1$,
as other cases can be treated similarly. Since $w^j_n$
collected in $\widetilde{f^1_n}$ has $\rho^j_n$ which is not
orthogonal to $\rho^1_n$, we have
\begin{align}\label{non-ortho}
\limsup_{n \rightarrow \infty} \left(\frac {\rho^j_n}{\rho^1_n} +
\frac {\rho^1_n}{\rho^k_n}\right) < \infty. \end{align} Moreover, by
construction, we have $|G^j_n(\widehat{w^j_n})| \leq
C_\delta \chi_{A_{1/2,1}}$. Here $G^j_n(\psi)(\xi) = (\rho_n^j)^\frac d2 \psi(\rho_n^j \xi)$. From
$G^1_n(\widehat{w^j_n}) = G^1_n(G^j_n)^{-1}G^j_n(\widehat{w^j_n})$
and $G^1_n(G^j_n)^{-1}\psi(\xi) = \left(\frac
{\rho^1_n}{\rho^j_n}\right)^{\frac d2}\psi\left(\frac
{\rho^1_n}{\rho^j_n} \xi\right)$, and the non-orthogonality
\eqref{non-ortho}, there exist $R_1$ and $R_2$ with
$0 < R_1 < R_2$ such that $|G^1_n(\widehat{w^j_n})| \leq
\widetilde{C_\delta} \chi_{A_{R_1, R_2}}$ for all $w_n^j$ collected
in $\widetilde{f_n^1}$.
This completes the proof of Proposition
\ref{prop-decom1}.
\end{proof}

In the next step, we further decompose $\{f_n^j\}$ into time
translated profiles.
\begin{prop}\label{further-decomp}
Suppose that $(f_n) \subset L^2_\rho H^\gamma_\sigma$ for some
$\ga \ge 0$ with $(\rho_n)^{d/2}|\widehat{D^\ga_\sigma
f_n}(\rho_n\xi)| \leq \widehat{F}(\xi)$, where $\widehat{F} \in
L^\infty(K)$ for some compact set $K \subset A = \{\xi : 0 < R_1 <
|\xi| < R_2\}$. Then there exist a family $(s^\ell_n)_{\ell \geq 1}
\subset \mathbb{R}$ and a sequence $(\phi^\ell)_{\ell \geq 1}
\subset L^2_\rho H^\gamma_\sigma$ satisfying the following:
\begin{enumerate}
\item[$1)$] For $\ell \neq \ell'$
$$\limsup_{n \rightarrow \infty} |s^\ell_n - s^{\ell'}_n| = \infty.$$
\item[$2)$]
For every  $M \geq 1$, there exists $(e^M_n) \subset L^2_\rho
H^\ga_\sigma$ such that $$f_n(x) = \sum_{\ell=1}^M
(\rho_n)^{d/2}(U(s^\ell_n) \phi^\ell)(\rho_n x) + e^M_n(x), \qquad
\limsup_{\substack{ M \to \infty \\ n\to \infty  }}  \|U(\cdot)
e^M_n\|_{L^q_tL^r_\rho H^{\ga + \hga}_\sigma} = 0.$$
\item[$3)$]
For any $M \geq 1$,
$$\limsup_{n \rightarrow \infty}\left(\|f_n\|_{L_\rho^2H_\sigma^\ga}^2 - (\sum_{\ell=1}^M
\|\phi^\ell\|_{L_\rho^2H_\sigma^\ga}^2 + \|e^M_n\|_{L_\rho^2H_\sigma^\ga}^2)\right) = 0.$$
\end{enumerate}
\end{prop}

\begin{proof}
Denote by $\mathcal F$ the collection of functions $ (F_n)_{n
\geq 1}$ which are given by $\widehat{F_n}(\xi) =
(\rho_n)^{d/2}\widehat{f_n}(\rho_n \xi)$, and define
$$\mathcal{W}(\mathcal F) = \{\text{\small weak-lim}\
U(-s^1_n)F_n(x) \text{ in } L^2_\rho H^\gamma_\sigma : s^1_n \in
\mathbb{R} \}, \qquad \mu(\mathcal F) = \sup_{\phi \in
\mathcal{W}(\mathcal F)} \|\phi\|_{L^2_\rho H^\gamma_\sigma}.$$
Here by the weak limit $\psi$ of $\psi_n$ in $L^2_\rho
H^\gamma_\sigma$ we mean that $\langle \psi_n - \psi,
\phi\rangle_{L_\rho^2H_\sigma^\ga} \equiv \langle D_\sigma^\gamma
(\psi_n-\psi), D_\sigma^\ga \phi\rangle_{L^2} \rightarrow 0$ as $n \rightarrow \infty$ for any $\phi \in
L_\rho^2H_\sigma^\ga$. Then $\mu(\mathcal F) \leq \limsup_{n
\rightarrow \infty} \|F_n\|_{L^2_\rho H^\gamma_\sigma}$.

{We may assume that $\mu(\mathcal F) > 0$, otherwise we are done by
using by a forthcoming inequality \eqref{bound-mu}.} Let us choose subsequences
$(F_n)_{n=1}^\infty$, $(s^1_n)$ and $\phi^1$ such that
$U(-s^1_n)F_n(x) \rightharpoonup \phi^1$ as $n \rightarrow \infty$
and $\|\phi^1\|_{L^2_\rho H^\gamma_\sigma} \geq \frac 12\mu(\mathcal
F)$. Let $F^1_n = F_n - U(s^1_n)\phi^1(x)$ and $\mathcal F^1 =
\{F^1_n\}_{n=1}^\infty$. Then
\begin{align*}
&\limsup_{n \rightarrow \infty} \|F^1_n\|^2_{L^2_\rho
H^\gamma_\sigma}= \limsup_{n \rightarrow \infty} \langle F_n - U(s^1_n)\phi^1(x),F_n - U(s^1_n)\phi^1(x)\rangle_{L^2_\rho H^\gamma_\sigma}\\
&= \limsup_{n \rightarrow \infty} \langle U(-s^1_n)F_n(x) - \phi^1,U(-s^1_n)F_n(x) - \phi^1\rangle_{L^2_\rho H^\gamma_\sigma}\\
&= \limsup_{n \rightarrow \infty} \left(\langle F_n, F_n \rangle_{L^2_\rho H^\gamma_\sigma} - \langle U(-s^1_n)F_n(x), \phi^1\rangle_{L^2_\rho H^\gamma_\sigma} - \langle \phi^1,U(-s^1_n)F_n(x) \rangle_{L^2_\rho H^\gamma_\sigma} + \langle \phi^1,\phi^1 \rangle_{L^2_\rho H^\gamma_\sigma}\right)\\
&= \limsup_{n \rightarrow \infty} \|F_n\|^2_{L^2_\rho H^\gamma_\sigma} - \|\phi^1\|^2_{L^2_\rho H^\gamma_\sigma}.
\end{align*}
Repeat the process with $F^1_n$ to get $s_n^2,\phi^2,F^2_n$ and so on.
By taking a diagonal sequence we may write
$$F_n(x) = \sum_{\ell=1}^M U(s_n^\ell)\phi^\ell(x) + F^M_n,$$
and we have
$$ \limsup_{n \rightarrow \infty} \|F_n\|_{L^2_\rho H^\gamma_\sigma}^2 = \sum_{\ell=1}^M \|\phi^\ell\|_{L^2_\rho H^\gamma_\sigma}^2 + \limsup_{n \rightarrow \infty} \|F_n^M\|_{L^2_\rho H^\gamma_\sigma}^2.$$
Thus $\sum_{\ell=1}^M \|\phi^\ell\|_{L^2_\rho H^\gamma_\sigma}^2$
converges. Hence we obtain $\limsup_{\ell \rightarrow
\infty}\|\phi^\ell\|_{L^2_\rho H^\gamma_\sigma} = 0$ and by
$\mu(\mathcal F^M) \leq 2 \|\phi^{M+1}\|_{L^2_\rho H^\gamma_\sigma}$
we get $\limsup_{M\rightarrow\infty} \mu(\mathcal F^M) = 0.$

Now we define $e_n^M$ by setting \[\rho_n^\frac d2
\widehat{e_n^M}(\rho_n\xi)=\widehat F_n^M .\] Then we are left
to show
\begin{align}\label{bound-mu}
\limsup_{n \rightarrow \infty} \|U(\cdot) e^M_n \|_{L^q_tL^r_\rho
H^{\ga + \hga}_\sigma} \lesssim \mu(\mathcal F^M)^\theta
\end{align}
for some $\theta$ with $0 < \theta < 1$.
By construction, we may assume that $\widehat{D^\ga_\sigma\phi^\ell}_{1\leq \ell \leq M}$ has common compact support $K$.
Invoking that the pair $(q, r)$ is $\alpha$-admissible, we get
\begin{align*}
&\|U(\cdot)e^M_n\|_{L^q_tL^r_\rho H^{\ga + \hga}_\sigma} 
\leq (\|U(\cdot)F^M_n\|_{L_t^{\widetilde q}L_\rho^{\widetilde r}H^{\ga+\hga}_\sigma})^{\frac {\widetilde q(6 - q)}{q(6 - \widetilde q)}} (\|U(\cdot)F^M_n\|_{L^6_tL^6_\rho H^{\ga + \frac {d-2^+}6}_\sigma})^{1-\frac {\widetilde q(6 - q)}{q(6 - \widetilde q)}}\\
\leq &(\|U(\cdot)F^M_n\|_{L_t^{\widetilde q}L_\rho^{\widetilde
r}H^{\ga+\hga}_\sigma})^{\frac {\widetilde q(6 - q)}{q(6 -
\widetilde q)}} (\|U(\cdot)F^M_n\|_{L^4_tL^4_\rho H^{\ga + \frac
{d-2^+}4}_\sigma})^{(1-\frac {\widetilde q(6 - q)}{q(6 - \widetilde
q)})\frac 23}\cdot(\|U(\cdot)F^M_n\|_{L^\infty_tL^\infty_\rho
H^{\ga}_\sigma})^{(1-\frac {\widetilde q(6 - q)}{q(6 - \widetilde
q)})\frac 13}
\end{align*}
for some $(\widetilde q, \widetilde r)$ satisfying $\frac n2 - \frac
{n}{\widetilde r} - \frac {\alpha}{\widetilde q} > 0$, $\frac
{1}{\widetilde q} \leq (d-1)(\frac 12 - \frac {1}{\widetilde r})$
and $(\frac 1q, \frac 1r) = \widetilde \theta(\frac {1}{\widetilde
q}, \frac {1}{\widetilde r}) + (1 - \widetilde \theta)(\frac
16,\frac 16)$ for some $0 < \widetilde \theta <1$. Concerning the
first factor, from Lemma \ref{str-radial} we have
$$\|U(\cdot)F^M_n \|_{L_t^{\widetilde q}L_\rho^{\widetilde r}H_\sigma^{\ga + \hga}} \lesssim  R_1^{\frac d2 - \frac{2d}{\widetilde q (2d-a)} - \frac{d}{\widetilde r}} \|F_n^M\|_{H^{\ga}_\sigma} \lesssim R_1^{\frac d2 - \frac{2d}{\widetilde q (2d-a)} - \frac{d}{\widetilde r}},$$
and Lemma \ref{44-str} gives
\[\|U(\cdot)F^M_n\|_{L^4_tL^4_\rho H^{\ga + \frac {d-2^+}{4}}_\sigma} \lesssim R_2^{\beta(\alpha,4,4)} \|F^M_n\|_{H^\ga_\sigma} \lesssim R_2^{\beta(\alpha,4,4)}.\]
Thus for \eqref{bound-mu} it suffices to show $\limsup_{n
\rightarrow \infty} \|U(t)F_n^M\|_{L^\infty_t L^\infty_\rho H^{\ga}_\sigma} \lesssim
\mu(\mathcal F^M)$. For this we may assume that there exists
$\delta
> 0$ such that $$\limsup_{\substack{ M \to \infty \\ n\to \infty  } }\|U(t)F^M_n\|_{L^\infty_t L^\infty_\rho H^{\ga}_\sigma} >
\delta.$$ Let $(s^M_n, \rho^M_n)$ be a pair such that  such that
$\frac12\|U(t) F^M_n\|_{L^\infty_t L^\infty_\rho H^{\ga}_\sigma} \le
\|U(s^M_n)(F^M_n)(\rho^M_n\cdot)\|_{H^{\ga}_\sigma}$. Then we show
that $(\rho^M_n)$ is uniformly bounded.
Let us first observe that for any $\rho_1, \rho_2 \in \mathbb R$
\begin{align*}
 &\big|\|D^\ga_\sigma U(t)(F^M_n)(\rho_2\cdot)\|_{L^2_\sigma} - \|D^\ga_\sigma U(t)(F^M_n)(\rho_1\cdot)\|_{L^2_\sigma}\big|\\
 &\leq \|D^\ga_\sigma U(t)F^M_n(\rho_2\cdot) - D^\ga_\sigma U(t)F^M_n(\rho_1\cdot)\|_{L^2_\sigma} \lesssim \|D^\ga_\sigma U(t)F^M_n(\rho_2\cdot) - D^\ga_\sigma U(t)F^M_n(\rho_1\cdot)\|_{L^\infty_\sigma} \\
 &\lesssim \sup_{x}|\nabla(D^\ga_\sigma U(t)F^M_n(x))||\rho_2 - \rho_1| \lesssim \int |\xi||e^{it|\xi|^\alpha}\widehat{D^\ga_\sigma F^M_n}(\xi)|d\xi|\rho_2 - \rho_1| \lesssim |\rho_2 - \rho_1|.
\end{align*}
 From this, we deduce that $\|U(s^M_n) F^M_n(\rho)\|_{H^{\ga}_\sigma} > \frac \delta2$ if $|\rho-\rho^M_n|
 \leq c\frac \delta4$ for some small constant $c > 0$. Taking $L_\rho^2$-norm on the set $\{|\rho^M_n| -
 c\frac \delta2 < |\rho| < |\rho^M_n| + c\frac \delta2\}$,
 we have $\frac \delta2 |\rho^M_n|^{n-1}  \frac \delta2 c \leq \|F^M_n\|_2 \leq 1$,
 which implies the uniform boundedness of $(\rho^M_n)$.

 Since $(\rho^M_n)$ is uniformly bounded, there exists $\rho^M_0$ such that
 $\rho^M_n \rightarrow \rho^M_0$ as $n \rightarrow \infty$, after taking a subsequence if necessary.
 Then for large $n$, we have $$\|U(s^M_n)(F^M_n)(\rho^M_0)\|_{H^{\ga}_\sigma}
 \geq \frac 12 \|U(s^M_n)(F^M_n)(\rho^M_n)\|_{H^{\ga}_\sigma}.$$ Let us choose  $\sigma^M_n \in S^{d-1}$
 such that $|D^{\ga}_\sigma U(s^M_n)(F^M_n)(\rho^M_0 \sigma^M_n)|\ge
 \frac34
 \|D^{\ga}_\sigma U(s^M_n)(F^M_n)(\rho^M_0 \cdot)\|_{L^\infty_\sigma}$. Since $S^{d-1}$ is compact,
 $\sigma^M_n \rightarrow \sigma^M_0$ as $n \rightarrow \infty$ for some $\sigma_0^M \in S^{d-1}$. Then for large $n$, we have
 $|D^\ga_\sigma U(s^M_n)(F^M_n)(\rho^M_0 \sigma^M_0)|  \geq \frac 12 |D^\ga_\sigma U(s^M_n)(F^M_n)(\rho^M_0 \sigma^M_n)|$. Set $\psi \in  C^\infty_0(\mathbb{R}^d)$ be such that $\psi
 = 1$ on $K$ and $\psi^M$ be a Schwartz function such that $\widehat{D^\ga_\sigma \psi^M} =
 \psi \widehat{\delta_{0}}$, where $\delta_{0}$ is Dirac-delta measure. Then we
 have
 \begin{align*}
 &\limsup_{n \rightarrow \infty} \|U(t) F^M_n\|_{L^\infty_t L^\infty_\rho H^\ga_\sigma} \lesssim \limsup_{n \rightarrow \infty}
 |D^\ga_\sigma U(s^M_n)(F^M_n)(\rho^M_0 \sigma^M_0)|\\
 &\lesssim   \limsup_{n \rightarrow \infty} |\langle U(s^M_n)(F^M_n)(y), \psi^M\rangle_{L_\rho^2H_\sigma^\ga}|
 \leq \mu(\mathcal F^M)  \|\psi^M\|_{L^2_\rho H^\ga_\sigma}\lesssim \mu(\mathcal F^M).
 \end{align*}
This completes the proof of Proposition \ref{further-decomp}.
\end{proof}

Now we are ready to prove Theorem \ref{main}.
\subsection{Proof of Theorem \ref{main}}\label{s4} We begin with a
preliminary decomposition. From Propositions \ref{prop-decom1} and
\ref{further-decomp}, we have
\begin{align}\label{pre-decomp}
u_n = \sum_{j = 1}^N \sum_{\ell = 1}^{M_j}\Phi_n^{\ell, j}  + \omega_n^{N,M_1,\cdots,M_N},
\end{align}
where \begin{align*}
&\qquad\qquad\qquad\Phi_n^{\ell, j} = U(t^{\ell,j}_n)[(h^j_n)^{-d/2}\phi^{\ell,j}(\cdot/h^j_n)],\\
&(h^j_n, t^{\ell,j}_n) = ((\rho^j_n)^{-1},(\rho^j_n)^{-\alpha}s^{\ell,j}_n),\quad \omega_n^{N,M_1,\cdots,M_N} = \sum_{j=1}^N e^{j, M_j}_n + q^N_n.
\end{align*}
Then we have
\begin{enumerate}
\item the orthogonality of parameter family $(h^j_n,t^{\ell, j}_n)$ (the property (2) in Theorem
\ref{main}),
\item the asymptotic orthogonality. i.e.
\begin{align*}
\|u_n\|_{L_\rho^2H_\sigma^\ga}^2 
= \sum_{j=1}^N\sum_{\ell=1}^{M_j}\|\phi^{\ell,j}\|_{L_\rho^2H_\sigma^\ga}^2 + \|\omega^{N,M_1,\cdots,M_N}_n\|_{L_\rho^2H_\sigma^\ga}^2 + o_n(1)
\end{align*}
and $\|\omega^{N,M_1,\cdots,M_N}_n\|_{L_\rho^2H_\sigma^\ga}^2 = \sum_{j=1}^N \|e^{j,M_j}_n\|_{L_\rho^2H_\sigma^\ga}^2 + \|q^N_n\|_{L_\rho^2H_\sigma^\ga}^2 $.  
\end{enumerate}
We will show that
$U(t)\,\omega_n^{N,M_1,\cdots,M_N}$ converges to zero in a Strichartz norm, i.e.
\begin{align}\label{str-err}\limsup_{n \rightarrow \infty} \|U(t)\, \omega_n^{N,M_1,\cdots,M_N}\|_{L^q_tL^r_\rho H^{\ga + \hga}_\sigma} \rightarrow 0 \text{ as }\min\{N, M_1,\cdots, M_N\} \rightarrow \infty, \end{align}
where $(q, r)$ is an $\alpha$-admissible pair for $\frac{d}{d-1} <
\alpha < 2$. We enumerate the pair $(\ell, j)$ by an order function
$\mathfrak n$ satisfying
\begin{center}
$\mathfrak n(\ell, j) < \mathfrak n(\ell', k)$ if $\ell + j < \ell' + k$ or $\ell + j = \ell' + k$ and $j < k$.
\end{center}
After relabeling, we get
$$u_n = \sum_{1 \leq j \leq l}U(t^j_n)[(h^j_n)^{-d/2}\phi^j(\cdot/h^j_n)](x) + \omega^l_n$$
where $\omega^l_n = u_M^{N,M_1,\cdots,M_n}$ with $l = \sum_{j=1}^N M_j.$ Then the proof is completed by \eqref{str-err}.

Now let us prove \eqref{str-err}. Given $\varepsilon > 0$, we take a
positive number $\Lambda$ such that for  $N \ge \Lambda$,
$\limsup_{n \rightarrow \infty} \|U(t)\, q^N_n \|_{L^q_tL^r_\rho
H^{\ga + \hga}_\sigma} \leq \varepsilon/3.$ Then for $N \geq \Lambda$,
we can find $\Lambda_N$ such that whenever $M_j \geq \Lambda_N$,
$\limsup_{n \rightarrow \infty} \|U(t)\, e^{j,M_j}_n
\|_{L^q_tL^r_\rho H^{\ga + \hga}_\sigma} \leq \varepsilon/3N $ for
$1\leq j \leq N.$ Now we rewrite $\omega_n^{N, M_1, \cdots, M_N}$ by
$$\omega_n^{N,M_1,\cdots,M_N} = q^M_n + \sum_{1 \leq j \leq N} e^{j, M_j \vee \Lambda_N}_n + R^{N, M_1, \cdots, M_n}_n,$$
where $M_j \vee \Lambda_N$ denotes $\max \{M_j, \Lambda_N\}$ and
\begin{align*}
R^{N,M_1,\cdots,M_N}_n = \sum_{1 \leq j \leq N} (e^{j,M_j}_n -
e^{j,\Lambda_N}_n) = \sum_{\substack{ 1 \leq j \leq N\\ M_j <
\Lambda_N}}\, \sum_{M_j < \ell < \Lambda_N}\Phi_n^{\ell, j}.
\end{align*}
Then we have
$$\lim_{n \rightarrow \infty} \|U(t) \omega_n^{N,M_1, \cdots, M_N}\|_{L^q_tL^r_\rho H^{\ga + \hga}_\sigma} \leq \frac {2\varepsilon}3 + \lim_{n \rightarrow \infty} \|U(t)R^{N,M_1,\cdots,M_N}_n\|_{L^q_tL^r_\rho H^{\ga + \hga}_\sigma}.$$
In order to handle the last term, we need the following lemma.
\begin{lem}\label{ortho}
For every $N, M_1, \cdots, M_N$, we have
\begin{align}\label{orth nonlinear1}\limsup_{n \rightarrow \infty} \| \sum_{j=1}^N\sum_{\ell = 1}^{M_j}U(t)\Phi_n^{\ell, j}\|^2_{L^q_tL^r_\rho H^{\ga + \hga}_\sigma} \le \sum_{j=1}^N\sum_{\ell = 1}^{M_j}\limsup_{n \to \infty}\left\|U(t)\Phi_n^{\ell, j}\right\|^2_{L^q_tL^r_\rho H^{\ga + \hga}_\sigma}.\end{align}
\end{lem}

\begin{proof}[Proof of Lemma \ref{ortho}] It suffices to show that for $(j, \ell) \neq (k, \ell')$,
\begin{align}\label{orth nonlinear2}\limsup_{n \rightarrow \infty} \|D_\sigma^{\ga + \hga}U(t)\Phi_n^{\ell, j}\; D_\sigma^{\ga + \hga}U(t)\Phi_n^{\ell', k}\|_{L_t^\frac q2 L_\rho^\frac r2 L^1_\sigma} = 0.\end{align}
When $(j, \ell) \neq (k, \ell')$, there are two possibilities: 
 \begin{enumerate}
 \item $\limsup_{n \rightarrow \infty} \left(\frac {h^k_n}{h^j_n} + \frac {h^j_n}{h^k_n}\right) = \infty$,
 \item $(h^j_n)=(h^k_n)$ and $\limsup_{n\rightarrow\infty}\frac{|t^{\ell, j}_n - t^{\ell', k}_n|}{(h^j_n)^\alpha} =
 \infty$.
 \end{enumerate}
More generally, we will prove that if $D^{\ga + \hga}_\sigma \Psi_1,
D^{\ga + \hga}_\sigma \Psi_2 \in L^q_tL^r_\rho L^2_\sigma$\footnote{Note that $D_\sigma$ commutes dilation.}, then
 $$\limsup_{n \rightarrow \infty} \Big\|\frac 1{(h^j_n)^{\frac d2}} D^{\ga + \hga}_\sigma \Psi_1(\frac {t - t^{\ell,j}_n}{(h^j_n)^\al}, \frac x{h^j_n}) \frac 1{(h^k_n)^{\frac d2}} D^{\ga + \hga}_\sigma \Psi_2(\frac {t-t^{\ell',k}_n}{(h^k_n)^\al},\frac x{h^k_n})\Big\|_{L^{\frac q2}_tL^{\frac r2}_\rho L^1_\sigma} =  0.$$
 By density argument, it suffices  to show this for $D^{\ga + \hga}_\sigma \Psi_1,
D^{\ga + \hga}_\sigma \Psi_2 \in C^\infty_0(\mathbb{R} \times
\mathbb{R}^d)$. Using the
 H\"{o}lder inequality and scaling in spatial variables,
\begin{align*}
& A_n := \Big\|\frac 1{(h^j_n)^{\frac d2}} D^{\ga + \hga}_\sigma
\Psi_1(\frac {t - t^{\ell, j}_n}{(h^j_n)^\al}, \frac x{h^j_n})
\frac 1{(h^k_n)^{\frac d2}} D^{\ga + \hga}_\sigma \Psi_2(\frac {t-t^{\ell', k}_n}{(h^k_n)^\al},\frac x{h^k_n})\Big\|_{L^{\frac q2}_tL^{\frac r2}_\rho L^1_\sigma}\\
& \leq \Big\|\Big\|\frac 1{(h^j_n)^{\frac d2}} D^{\ga + \hga}_\sigma
\Psi_1(\frac {t - t^{\ell, j}_n}{(h^j_n)^\al}, \frac
x{h^j_n})\Big\|_{L^r_\rho L^2_\sigma} \Big\|\frac 1{(h^k_n)^{\frac
d2}} D^{\ga + \hga}_\sigma \Psi_2(\frac {t-t^{\ell',
k}_n}{(h^k_n)^\al},\frac x{h^k_n})\Big\|_{L^r_\rho L^2_\sigma}\Big\|_{L^{\frac q2}_t}\\
& \leq \Big\|\frac 1{(h^j_n)^{\frac d2 - \frac dr}}\Big\| D^{\ga +
\hga}_\sigma \Psi_1(\frac {t - t^{\ell, j}_n}{(h^j_n)^\al},
x)\Big\|_{L^r_\rho L^2_\sigma} \frac 1{(h^k_n)^{\frac d2 - \frac
dr}}\Big\|D^{\ga + \hga}_\sigma \Psi_2(\frac {t-t^{\ell',
k}_n}{(h^k_n)^\al}, x)\Big\|_{L^r_\rho L^2_\sigma}\Big\|_{L^{\frac
q2}_t}.
\end{align*}
Then by time translation and scaling in time, we estimate
\begin{align*}
A_n &= \Big\|\frac 1{(h^j_n)^{\frac d2 - \frac dr - \frac {2\al}q}}\| D^{\ga + \hga}_\sigma \Psi_1(t, x)\|_{L^r_\rho L^2_\sigma} \frac 1{(h^k_n)^{\frac d2 - \frac dr}}\|D^{\ga + \hga}_\sigma \Psi_2((\frac {h^j_n}{h^k_n})^\al t - \frac{t^{\ell', k}_n-t^{\ell, j}_n}{(h^k_n)^\al}, x)\|_{L^r_\rho L^2_\sigma}\Big\|_{L^{\frac q2}_t}\\
&\leq \Big\|(\frac {h^j_n}{h^k_n})^{\frac \al q}\Big\|D^{\ga +
\hga}_\sigma \Psi_1(t, x)\Big\|_{L^r_\rho L^2_\sigma} \Big\| D^{\ga
+ \hga}_\sigma \Psi_2((\frac {h^j_n}{h^k_n})^\al t - \frac{t^{\ell',
k}_n-t^{\ell, j}_n}{(h^k_n)^\al},x)\Big\|_{L^r_\rho
L^2_\sigma}\Big\|_{L^{\frac q2}_t}.
\end{align*}
Since the support in time of $\|D^{\ga + \hga}_\sigma \Psi_1(t,
\cdot)\|_{L^r_\rho L^2_\sigma}$ is compact, from the above condition
(1) or (2) it follows that $\limsup_{n \to \infty} A_n = 0$. This
completes the proof of Lemma \ref{ortho}.
\end{proof}

By Lemma \ref{ortho} and the Strichartz estimates (Lemma
\ref{str-radial}) it follows that
\begin{align*}
\limsup_{n \rightarrow
\infty}\|U(t)R_n^{N,M_1,\cdots,M_N}\|_{L^q_tL^r_\rho H^{\ga +
\hga}_\sigma}^2 &\le \sum_{\substack{1 \leq j \leq N\\
M_j < \Lambda_N}}\,\sum_{M_j < \ell < \Lambda_N} \limsup_{n
\rightarrow \infty}\|U(t)\Phi_n^{\ell, j}\|_{L^q_tL^r_\rho H^{\ga + \hga}_\sigma}^2\\
&\lesssim \sum_{1 \leq j \leq N} \sum_{\ell > M_j} \|\phi^{\ell,j}\|_{L^2_\rho H^\ga_\sigma}^2.
\end{align*}
Since $\sum_{j, \ell} \|\phi^{\ell,j}\|_{L^2_\rho H^\ga_\sigma}^2$ is convergent, we have
$$\limsup_{n \rightarrow \infty} \Big(\sum_{j=1}^N\sum_{\ell > M_j} \|U(t)\Phi_n^{\ell, j}
\|_{L^q_tL^r_\rho H^{\ga + \hga}_\sigma}^2\Big)^{\frac 12} \leq
\frac \varepsilon3,$$ provided that $\min(N, M_1, \cdots, M_N\}$ is
sufficiently large. This completes the proof of Theorem \ref{main}.
\vspace{2mm}
\section{Application: Blowup phenomena}
In this section, we present applications of linear profile
decomposition to the mass-critical Hartree equations \eqref{eqn}.
Almost all parts of the proofs of results are very similar to the radial case. So, we
omit them and refer readers to \cite{chkl2} for them.  Firstly, we  get
nonlinear profile decompositions of the solutions to \eqref{eqn}.

\subsection{\textbf{Nonlinear profile decomposition}}
Let us set
\begin{equation}\label{numbers} (q_\circ,r_\circ)=\Big(3,\frac{6d}{3d -
2\al}\Big),\quad \ga_1 = \frac {d^2 - \alpha d + \alpha}{4d},
\quad\ga_2 = \frac {d - 1 + \ga_1}3, \quad \tga = \ga_2 - \ga_1 +
\ga.
\end{equation} As it will be shown in Appendix \ref{s-gwp} by the
usual fixed point argument and the Strichartz estimate in Lemma
\ref{str-radial}, the local well-posedness theory can be based on
the estimate of space-time norm $\|u\|_{L_t^\qq L_\rho^\rr
H_\sigma^{\tga}(I \times\mathbb{R}^d)}$. For a given sequence of
angularly regular data $ (u^0_n) \subset L^2_\rho H_\sigma^\ga $,
using the linear profile decomposition (Theorem~\ref{main}), we have
sequences $(\phi^j)_{1\leq j \leq l} \in L^2_\rho H_\sigma^\ga$,
$\omega_n^l \in L^2_\rho H_\sigma^\ga$, $(h_n^j, t_n^j)_{1 \leq j
\leq l, n \geq 1}$ which satisfy (1)--(3) in Theorem~\ref{main}.
Then by taking subsequence, if necessary, we may assume that $ t^j
\in \{-\infty, 0, \infty\}$. Here we denote $t^j = \lim_n t^j_n$.
Using the local well-posedness theorem with initial data at $t=0$ or
$t=\pm \infty $ (see Proposition~\ref{wpasy} below), we define the
nonlinear profile by the maximal nonlinear solution for each linear
profile.


\begin{defn}
Let $(h_n,t_n) $ be a family of parameters and $(t_n)$ have a
limit in $[-\infty, \infty]$. Given a linear profile $\phi \in
L^2_\rho H_\sigma^\ga$ with $(h_n,t_n)$, we define the nonlinear
profile associated with them to be  the maximal solution $\psi$ to
\eqref{eqn} which is in $C_tL^2_\rho
H_\sigma^\ga((-T_{\min},T_{\max}) \times \mathbb{R}^d)$ satisfying
that
$$\lim_{n \to \infty}\|U(t_n)\phi - \psi(t_n)\|_{L^2_\rho H_\sigma^\ga} = 0.$$
Here $(-T_{\min},T_{\max})$ is the maximal existence time interval.
\end{defn}

 Then, the linear profile decomposition yields the
nonlinear profile decomposition. It is the key tool for proving
blowup phenomena in what follows.
\begin{thm}\label{nlpf}
Let $(u_n^0) \subset L^2_\rho H_\sigma^\ga$ be a bounded sequence.
Suppose that $(\phi^j)_{1\leq j \leq l}  \subset L^2_\rho
H_\sigma^\ga$, $\omega_n^l \in L^2_\rho H_\sigma^\ga$, and $(h_n^j,
t_n^j)_{1 \leq j \leq l, n \geq 1}$ are sequences obtained from
Theorem~\ref{main}. Let $u_n \in C_tL_\rho^2H_\sigma^\gamma(J_n \times \mathbb R^d)$ be
the maximal solution of \eqref{eqn} with initial data $u_n(0) =
u_n^0$. For each $j\ge 1$, suppose $(\psi^j)_{1\leq j
\leq l} \subset C_tL^2_\rho H_\sigma^\ga((-T^j_{\min},T^j_{\max})
\times \mathbb{R}^d)$ is the maximal nonlinear profile associated
with $(\phi^j)_{1\leq j \leq l}$ and $(h_n^j, t_n^j)_{1 \leq j \leq l,
n \geq 1}$. Let $(I_n)$ be a family of nondecreasing time
intervals containing $0$. Then, the following two are equivalent;
\begin{enumerate}
\item $\limsup_{n \rightarrow \infty} \|\Ga^j_n \psi^j \|_{L_t^\qq L_\rho^\rr H_\sigma^\tga(I_n \times \mathbb{R}^d)} < \infty, \quad j\ge1,$
\item $\limsup_{n \rightarrow \infty} \|u_n \|_{L_t^\qq L_\rho^\rr H_\sigma^\tga(I_n \times \mathbb{R}^d)} < \infty.$
\end{enumerate}
Here  $\Gamma^j_n\psi^j = \frac 1{(h^j_n)^{d/2}}\psi^j(\frac {t -
t^j_n}{(h^j_n)^\al},\frac x{h^j_n})$. Moreover, if $(1)$ or $(2)$
holds true, we have a decomposition
$$u_n = \sum_{j=1}^l \Ga^j_n \psi^j + U(\cdot)\omega^l_n + e^l_n$$
with $\lim_{l \to \infty}\limsup_{n \to \infty}
\big(\|\big(U(\cdot)\omega^l_n\|_{L_t^\qq L_\rho^\rr
H_\sigma^\tga(I_n \times \mathbb{R}^d)} + \|e^l_n\|_{L_t^\qq
L_\rho^\rr H_\sigma^\tga(I_n \times \mathbb{R}^d)} \big)= 0.$
\end{thm}

\subsection{Applications}\label{blow-ph}
We consider blowup phenomena of solutions to \eqref{eqn}. 
If the solution fails
to persist, then the space-time norm blows up. The blowup solution
is defined as follows.

\begin{defn}
A solution $u \in C_t L^2_\rho
H_\sigma^\ga((-T_{min},T_{max})\times\mathbb{R}^d)$ to \eqref{eqn}
is said to blow up if $\|u\|_{L_t^\qq L_\rho^\rr
H_\sigma^{\tga}((-T_{min},T_{max})\times\mathbb{R}^d)} = \infty$.
Here $(-T_{min}, T_{max})\in [-\infty,\infty]$ denotes the maximal
time interval of existence of the solution.
\end{defn}
Since $ T_{max}$ or $T_{min}$ may be $\infty$, we regard
non-scattering global solutions as blowup solutions at infinite
time. We also define a minimal quantity of solutions from which a
solution may ignite to blow up.
\begin{defn}\label{minimal mass} Define
\begin{align*}  \delta_0 = \sup\{A : \text{for any } &u_0 \text{ with } \|u_0\|_{L^2_\rho
H_\sigma^\ga} \le A,\,\,\eqref{eqn} \text{ is globally well-posed on } \mathbb{R}  \\
              &\text{ satisfying }  \|u\|_{L_t^\qq L_\rho^\rr
H^\tga_\sigma(\mathbb{R}\times\mathbb{R}^d)} < \infty  \}.   \end{align*}
\end{defn}
By the small data global existence (see Section \ref{s-gwp} below),
we have $\delta_0 >0$. Moreover, for any $ \delta >\delta_0 $ there
exists a blowup solution $u$ with $\delta_0 \le \|u_0\|_{L^2_\rho
H_\sigma^\ga} \le \delta$. Such a solutions satisfies that $\|u(t)\|_{L^2_\rho
H_\sigma^\ga} \ge \delta_0$ for all $t \in (-T_{min}, T_{max})$.
As opposed to $L_x^2$-norm, $L^2_\rho H^\ga_\sigma$ is not conserved in time. We only have a lower bound from the mass conservation law. One may compare $\delta_0$ with minimal mass of radial blowup solutions \cite{chkl2}. If we set
\begin{align*}  \delta_{0,rad} = \sup\{A : \text{for any radial data } &u_0 \text{ with } \|u_0\|_{L^2_x} \le A,\,\,\eqref{eqn} \text{ is globally well-posed on } \mathbb{R}  \\
              &\text{ satisfying }  \|u\|_{L_t^\qq L_x^\rr
(\mathbb{R}\times\mathbb{R}^d)} < \infty  \},  \end{align*}
then, it is clear that $ \delta_0 \le \delta_{0,rad}$.

Following the same lines of arguments which were used for the radial case (Theorem 1.6, \cite{chkl2}), we obtain the existence of minimal blowup solutions  with initial data in ${L^2_\rho
H_\sigma^\ga}$.
\begin{thm}\label{th:minimal blow up}
Assume $ \delta_0 <\infty$. Then, there exists a blowup solution $u$ to $\eqref{eqn}$ with initial
data $u_0 \in L^2_\rho H^\ga_\sigma$ such that $\|u_0\|_{L^2_\rho
H^\ga_\sigma} = \delta_0$.
\end{thm}
\noindent Note that the minimal blowup solution obtained above may not be radial, and so $\|u(t) \|_{L^2}$ may be smaller than $\delta_0$.

In view of the local theory in Appendix A there are two possible blowup scenarios:
\begin{itemize}
\item [(1)] $\sup_{t \in (-T_{min}, T_{max})}\|u(t)\|_{L_\rho^2H_\sigma^\gamma}  = \infty$,
\item [(2)] $\sup_{t \in (-T_{min}, T_{max})}\|u(t)\|_{L_\rho^2H_\sigma^\gamma} < \infty$ and $\|u\|_{L_t^\qq L_\rho^\rr H^\tga_\sigma
((-T_{min}, T_{max})\times\mathbb{R}^d)} = \infty$.
\end{itemize}
 We  focus on  the second scenario. In the case of (2), we deduce from the linear and nonlinear profile decompositions a compactness of the trajectory of solution $u(t)$ as in the radial case. If especially $\sup_{t \in (-T_{min}, T_{max})}\|u(t)\|_{L_\rho^2H_\sigma^\gamma}^2 <
{2}\delta_0^2$, 
then the blowup solution does not form more than one blowup
profile. Thus, this gives  a weaker form of compactness property of
the blowup solutions.
\begin{thm}\label{cpbs}
Let $u$ be finite time blowup solution of \eqref{eqn} at $T^*$ with
$\sup_{0 < t < T^*}\|u(t,\cdot)\|_{L^2_\rho H^\ga_\sigma} < \sqrt{2}\delta_0$ and let
$t_n \nearrow T^*$. Then there exist $\phi \in L^2_\rho
H^\ga_\sigma$ and $(h_n)_{n=1}^\infty$ satisfying $
\label{weakcon} h^{d/2}_n u(t_n, h_nx) \rightharpoonup \phi \text {
(weakly) in } L^2_\rho H^\ga_\sigma $ and if solution of \eqref{eqn}
with initial data $\phi$ blows up at $T^{**}$, then
\begin{equation}
\label{hhh} \lim_{n \rightarrow \infty} \frac {h_n}{(T^* -
t_n)^{1/\alpha}} \leq \frac 1{(T^{**})^{1/\alpha}} \end{equation}
up to subsequence.
\end{thm}

When a blowup occurs, only one profile blows up by shrinking in
scale.  See \cite{chkl2} for the radial case  and \cite{chle} for
related results when $\al > 2$. As a byproduct of Theorem
\ref{cpbs}, we obtain a concentration property of blowup solution as follows.

\begin{thm}\label{mass concentration}
Let $u$ be a finite time blowup solution at $T^*$ with
$\sup_{0 < t < T^*}\|u(t)\|_{L^2_\rho H^\ga_\sigma} < \sqrt{2}\delta_0$ and let  $t_n
\nearrow T^*$. Then for $\lambda(t_n)$ satisfying $\frac {(T^* -
t_n)^{1/\alpha}}{\lambda(t_n)} \rightarrow 0$
\begin{equation}\label{masscon}
\limsup_{n \rightarrow \infty} \int_{|x| \leq \lambda(t_n)}
|u(t_n,x)|^2 dx \geq \|\phi\|_{L^2_x}^2, \end{equation}
where $\phi$ is defined as in Theorem \ref{cpbs}.
\end{thm}


\vspace{5mm}

\appendix
\section{Local wellposedness and global well-posedness with small data}\label{s-gwp}
In this section we show local well-posedness and global
well-posedness with small data by the standard fixed point argument.
Local well-posedness and global well-posedness with small data. Let
$q_\circ,$ $r_\circ,$ $\ga_1$, $\ga_2$, and $\tga$ be given by
\eqref{numbers}.

\begin{prop}\label{CP1} For any $u_0 \in
L^2_\rho H^\ga_\sigma$, there exists a unique solution $u$ to
\eqref{eqn} such that $u \in C_tL^2_\rho H^\ga_\sigma((-T_{min},
T_{max}) \times \mathbb{R}^d) \cap L^{q_\circ}_{t,
loc}L^{r_\circ}_\rho H^{\tga}_\sigma((-T_{min}, T_{max}) \times
\mathbb{R}^d)$ whenever $$\ga \geq \frac {d^2 - \alpha d +
\alpha}{4d}.$$ Moreover, if $\|u_0\|_{L^2_\rho H^\ga_\sigma}$ is
sufficiently small, then $T_{min} = T_{max} = \infty$ and $u \in
L^{q_\circ}_{t}L^{r_\circ}_\rho H^{\tga}_\sigma(\mathbb{R}^{1+d})$.
\end{prop}

\begin{proof}
The proof is based on the standard fixed point argument. Let us
consider the integral equation
\begin{align}\label{int eqn}u(t) = U(t)u_0 + i\lambda\int_0^t U(t-s)((|x|^{-\alpha}*|u|^2)u)(s)ds\end{align}
on Banach space $X = X_{T,\,\mu}$ given by
\begin{align*}
X := \{ v \in L^{q_\circ}_tL^{r_\circ}_\rho H_\sigma^{\tga}([-T,T] \times \mathbb{R}^d) : \|v\|_{L^{q_\circ}_tL^{r_\circ}_\rho H_\sigma^{\tga}([-T,T] \times \mathbb{R}^d)} \leq \mu \}.
\end{align*}
For simplicity we denote $L^{q}_tL^{r}_\rho H_\sigma^{\al}([-T,T)
\times \mathbb{R}^d)$ by $L^{q}_TL^{r}_\rho H_\sigma^{\al}$. Then
$X$ is obviously a complete metric space with metric $d(u, v) = \|u
- v\|_{L^{q_\circ}_TL^{r_\circ}_\rho H_\sigma^{\tga}}$. To proceed
we consider nonlinear mapping $\mathcal{N}$ defined by
\[\mathcal N(v)(t) := U(t)u_0 + i\lambda\int_0^t U(t - s)(|x|^{-\alpha}*|v(s)|^2)(v(s))\,ds\]
and show that it is self mapping on $X$.

In fact, by Strichartz estimates (Lemma \ref{str-radial}), we have
\begin{align*}
\|\mathcal{N}(u) \|_{L_T^{q_\circ}L_\rho^{r_\circ}H_\sigma^{\tga}}
\lesssim \|U(\cdot)u_0
\|_{L_T^{q_\circ}L_\rho^{r_\circ}H_\sigma^{\tga}} +
\|D_\sigma^{\ga}((|x|^{-\al} *
|u|^2)u)\|_{L_T^{1}L_\rho^{2}L_\sigma^2}.
 \end{align*}
By Leibniz rule on the unit sphere and the H\"oder's inequality, one
obtain
\begin{align*}
\|D_\sigma^{\ga}((|x|^{-\al} *
|u|^2)u)\|_{L_T^{1}L_\rho^{2}L_\sigma^2}
 & \lesssim \||x|^{-\alpha}*|u|^2\|_{L_T^{q_{1}}L_\rho^{r_{1}}L_\sigma^{p_{1}}}\|D_\sigma^{\ga}u
 \|_{L_T^{q_{2}}L_\rho^{r_{2}}L_\sigma^{p_{2}}} \\ &\qquad\qquad+ \||x|^{-\alpha}*(D_\sigma^{\ga}|u|^2)\|_{L_T^{q_{1}}L_\rho^{r_{1}}L_\sigma^{p_{3}}}\|u\|_{L_T^{q_{2}}L_\rho^{r_{2}}L_\sigma^{p_{4}}},
\end{align*}
where
\begin{align*}\begin{aligned}\label{expo}
 &\frac 1{q_{1}} = \frac 23,\quad \frac 1{q_{2}} = \frac 13, \quad
\frac 1{r_{1}} = \frac 2{r_\circ} - \frac {d - \alpha}{d},\quad
\frac 1{r_{2}} = \frac 1 r_\circ,\quad \frac 1{p_{1}} = 2(\frac 12 - \frac {\ga_2}{d-1}),\\
&\frac 1{p_{2}} = \frac 12 - \frac {\ga_2 - \ga_1}{d-1}, \frac
1{p_{3}} = 1 - \frac {\ga_2}{d-1} - \frac {\ga_2 - \ga_1}{d-1},
\quad \frac 1{p_{4}} = \frac 12 - \frac {\ga_2}{d-1}.
\end{aligned}\end{align*}

To treat the convolution term we use the following lemma about
fractional integration in the space $L_\rho^rL_\sigma^p$.

\begin{lem}\label{frac-int} Let  $1 < r, \widetilde r < \infty$, $0 < \beta < n$, and $1
\le p \le \infty$. If $\frac1r = \frac1{\widetilde r} -
\frac{n-\beta}{n}$, then
\[
\||x|^{-\beta}*f\|_{L_\rho^rL_\sigma^p} \lesssim \|f\|_{L_\rho^{\widetilde r} L_\sigma^p}.
\]
\end{lem}

\begin{proof}[Proof of Lemma \ref{frac-int}]
We use the following pointwise estimate, which is shown in p.15 of
\cite{chonak}:

For any $\rho > 0$ and $\theta \in S^{n-1}$
\begin{align}\label{pointwise}
\int_{S^{n-1}}\left||x|^{-\beta}*f\right|(\rho\sigma)\,d\sigma \le
(|x|^{-\beta}*F)(\rho\theta),
\end{align}
where $F(\rho) = \int_{S^{n-1}}|f(\rho\sigma)|\,d\sigma$. By taking
$L_\rho^r$ on both sides of \eqref{pointwise}, from
Hardy-Littlewood-Sobolev inequality we get
$$
\||x|^{-\beta}*f\|_{L_\rho^rL_\sigma^1} \lesssim
\||x|^{-\beta}*F\|_{L_x^r} \lesssim \|F\|_{L_x^{\widetilde r}}
\lesssim \|f\|_{L_\rho^{\widetilde r}L_\sigma^1}.
$$
On the other hand, we have $ |\langle |x|^{-\beta}*f, g\rangle|  =
|\langle f, |x|^{-\beta}*g\rangle| \lesssim
\|f\|_{L_\rho^{\widetilde
r}L_\sigma^\infty}\||x|^{-\beta}*g\|_{L_\rho^{\widetilde
r'}L_\sigma^1} \lesssim \|f\|_{L_\rho^{\widetilde
r}L_\sigma^\infty}\|g\|_{L_\rho^{r'}L_\sigma^1}, $ which implies
that \[\||x|^{-\beta}*f\|_{L_\rho^rL_\sigma^\infty} \lesssim
\|f\|_{L_\rho^{\widetilde r}L_\sigma^\infty}.\] Interpolation
between these two estimates gives the desired estimates.
\end{proof}

Then by Lemma \ref{frac-int}, we get
\begin{align*}
&\||x|^{-\alpha}*|u|^2\|_{L_T^{q_{1}}L_\rho^{r_{1}}L_\sigma^{p_{1}}}\|D_\sigma^{\ga}u\|_{L_T^{q_{2}}L_\rho^{r_{2}}L_\sigma^{p_{2}}}  + \||x|^{-\alpha}*(D_\sigma^{\ga}|u|^2)\|_{L_T^{q_{1}}L_\rho^{r_{1}}L_\sigma^{p_{3}}}\|u\|_{L_T^{q_{2}}L_\rho^{r_{2}}L_\sigma^{p_{4}}}\\
&\qquad\lesssim
\||u|^2\|_{L_T^{q_{1}}L_\rho^{\frac {r_{2}} 2}L_\sigma^{p_{1}}}\|D_\sigma^{\ga}u\|_{L_T^{q_{1,2}}L_\rho^{r_{2}}L_\sigma^{p_{2}}} + \|D_\sigma^{\ga}|u|^2\|_{L_T^{q_{1}}L_\rho^{\frac {r_{2}} 2}L_\sigma^{p_{3}}}\|u\|_{L_T^{q_{2}}L_\rho^{r_{2}}L_\sigma^{p_{4}}}.
\end{align*}
Finally, the Leibniz rule and Sobolev embedding on the unit sphere
gives
\begin{align*}
&\||u|^2\|_{L_T^{q_{1}}L_\rho^{\frac {r_{2}} 2}L_\sigma^{p_{1}}}\|D_\sigma^{\ga}u\|_{L_T^{q_{2}}L_\rho^{r_{2}}L_\sigma^{p_{2}}} + \|D_\sigma^{\ga}|u|^2\|_{L_T^{q_{1}}L_\rho^{\frac {r_{2}} 2}L_\sigma^{p_{3}}}\|u\|_{L_T^{q_{2}}L_\rho^{r_{2}}L_\sigma^{p_{4}}}\\
&\quad\lesssim \|D_\sigma^{\ga_2} u \|_{L_T^{q_\circ}L_\rho^{r_\circ}L_\sigma^2}^2\|D_\sigma^{\tga} u \|_{L_T^{q_\circ}L_\rho^{r_\circ}L_\sigma^2} \lesssim \|D_\sigma^{\tga} u \|_{L_T^{q_\circ}L_\rho^{r_\circ}L_\sigma^2}^3 \lesssim \mu^3.
\end{align*}
Since $\|U(\cdot)u_0\|_{L_T^{q_\circ}L_\rho^{r_\circ}H_\sigma^{\tga}} \lesssim \|u_0\|_{L_\rho^2H_\sigma^\ga}$, for suitable $T$ and $\mu$ we have
$$
\|\mathcal N(u)\|_{L_T^{q_\circ}L_\rho^{r_\circ}H_\sigma^{\tga}} \le C(\|U(\cdot)u_0\|_{L_T^{q_\circ}L_\rho^{r_\circ}H_\sigma^{\tga}} + \mu^3) \le \mu.
$$

Similarly  one can easily show $u\to \mathcal N(u)$ is a contraction
map on $X$ for suitable $T$ and $\mu$. This implies that there
exists a unique solution $u \in
L_T^{q_\circ}L_\rho^{r_\circ}H_\sigma^{\tga}$ to \eqref{int eqn}.
Now using Lemma \ref{str-radial} again, we get
$$
\|\mathcal N(u)\|_{L_T^\infty L_\rho^2H_\sigma^{\ga}} \le \|u_0\|_{L_\rho^2H_\sigma^\ga} + C\mu^3.
$$
Thus $u \in L_T^\infty L_\rho^2H_\sigma^{\ga} \cap
L_T^{q_\circ}L_\rho^{r_\circ}H_\sigma^{\tga}$ and from the
uniqueness and the formula \eqref{int eqn} the well-posedness is
straightforward.

On the other hand, if $\|u_0\|_{L_\rho^2H_\sigma^\ga}$ and $\mu$ are
sufficiently small, then the functional $\mathcal N$ is shown to be
a contraction map on complete metric space $Y$ given by
$$
Y := \{ v \in (C_tL_\rho^2H_\sigma^\ga \cap L^{q_\circ}_tL^{r_\circ}_\rho H_\sigma^{\tga})(\mathbb{R}^{1+d}) : \|v\|_{(L_t^\infty L_\rho^2H_\sigma^\ga \cap L^{q_\circ}_tL^{r_\circ}_\rho H_\sigma^{\tga})(\mathbb{R}^{1+d})} \leq \mu \}.
$$
We omit the details.
\end{proof}

The well-posedness for a given asymptotic state is also similar and
fairly standard. We provide its proof for completeness.

\begin{prop}\label{wpasy}
Given $u_\infty \in  L_\rho^2H_\sigma^\gamma(\mathbb R^d)$, there exists a positive $T$
and a unique solution $u$ to \eqref{eqn} such that $u \in
C_tL^2_\rho H^\ga_\sigma([T,\infty) \times \mathbb{R}^d) \cap
L^{q_\circ}_tL^{r_\circ}_\rho H_\sigma^{\tga}([T,\infty) \times
\mathbb{R}^d)$ and $\|u(t) - U(t)u_\infty\|_{L^2_\rho H^\ga_\sigma}
\rightarrow 0$ as $t \rightarrow \infty.$
\end{prop}

\begin{proof}
We define a nonlinear mapping $\mathcal{N}$ by
$$\mathcal N(v)(t) := i\lambda\int_t^\infty U(t - s)(|x|^{-\alpha}*|U(s)u_\infty + v(s)|^2)(U(s)u_\infty + v(s))ds$$
for $v$ in Banach space $Z = Z_{T, \mu}$ given by
\begin{align*}
Z := \big\{ v \in &C_tL^2_\rho H^\ga_\sigma([T,\infty) \times
\mathbb{R}^d) \cap L^{q_\circ}_tL^{r_\circ}_\rho H^{\tga}([T,\infty)
\times \mathbb{R}^d) : \\&\qquad \|v\|_{L_t^\infty L^2_\rho
H^\ga_\sigma([T,\infty) \times \mathbb{R}^d)} +
\|v\|_{L^{q_\circ}_tL^{r_\circ}_\rho H^{\tga}([T,\infty) \times
\mathbb{R}^d)} \leq \mu \big\}.
\end{align*}
Similarly to proof of Proposition \ref{CP1}, one can get
\begin{align*}
&\|\mathcal N(v)\|_{L_t^\infty L^2_\rho H^\ga_\sigma([T,\infty)\times\mathbb{R}^d)} + \|\mathcal N(v)\|_{L^{q_\circ}_tL^{r_\circ}_\rho H^{\tga}([T,\infty)\times\mathbb{R}^d)} \\
&\quad \lesssim \|U(\cdot)u_\infty\|^{3}_{L^{q_\circ}_tL^{r_\circ}_\rho H^{\tga}([T,\infty)\times\mathbb{R}^d)} + \|v\|^{3}_{L^{q_\circ}_tL^{r_\circ}_\rho H^{\tga}([T,\infty)\times\mathbb{R}^d)}.
\end{align*}
Since $\|U(\cdot)u_\infty\|_{L^{q_\circ}_tL^{r_\circ}_\rho
H^{\tga}([T,\infty)\times\mathbb{R}^d)} \lesssim
\|u_\infty\|_{L^2_\rho H^\ga_\sigma}$ by Lemma \ref{str-radial},
$\mathcal N$ becomes a self-mapping on $Z$ for sufficiently large
$T$. Similarly one can easily show that $\mathcal N$ is a
contraction mapping on $Z$. Let $v$ be the fixed point of $\mathcal
N$ in $Z$. Then by continuity we get $\|v(t)\|_{L^2_\rho
H^\ga_\sigma} \rightarrow 0$ as $t \rightarrow \infty$.

Now we write $u(t)$ as $u(t) = U(t)u_\infty + v(t).$ Then it follows
that $\|u(t) - U(t)u_\infty\|_{L^2_\rho H^\ga_\sigma} \rightarrow 0
\text{ as } t \rightarrow \infty$. It remains to show that
\begin{align}\label{sol-u} u(\tau) = U(\tau - t)u(t) -i \lambda \int^\tau_t U(\tau - s)
((|x|^{-\alpha} * |u|^2)u)(s)ds.\end{align}
In fact, since $v(\tau) = \mathcal N(v)(\tau)$, one can show that
$$v(\tau) = U(\tau-t)v(t)-i\lambda\int_t^\tau U(\tau-s)((|x|^{-\alpha} * |u|^2)u)(s)\,ds.$$
Thus
$$
u(\tau) = U(\tau)u_\infty + v(\tau) = U(\tau - t)(U(t)u_\infty +
v(t)) - i\lambda \int_t^\tau U(\tau-s)((|x|^{-\alpha} *
|u|^2)u)(s)\,ds,
$$
which yields \eqref{sol-u}.
\end{proof}

\section*{Acknowledgments} Y. Cho was supported by NRF grant 2012-0002855 (Republic of Korea), S. Lee and G. Hwang were in part by NRF grant 2009-0083521 (Republic of Korea). S. Kwon was partially supported by TJ Park science fellowship and NRF grant 2010-0024017 (Republic of Korea). \medskip

%


\begin{thebibliography}{00}
\bibitem{bage} H. Bahouri, P. G\'erard, \emph{High frequency approximation of solutions
to critical nonlinear wave equations,} Amer. J. Math. \textbf{121} (1999),
no. 1, 131-175.

\bibitem{beva} P. B\'{e}gout and A. Vargas, {\it Mass concentration phenomena for the $L^2$-critical nonlinear Schr\"{o}dinger equation}, Trans. Amer. Math. Soc., \textbf{359} (11) (2007), 5257-5282.


\bibitem{bo1} J. Bourgain, \emph{Refinements of Strichartz' inequality and
applications to 2D-NLS with critical nonlinearity}, Int. Math. Res.
Not. \textbf{5} (1998) 253-283.

\bibitem{bul}  A. Bulut, \emph{Maximizers for the Strichartz Inequalities for the
Wave Equation,} Differential and Integral Equations 23 (2010) 1035-1072


\bibitem{ck} R. Carles and S. Keraani, {\it On the role of quadratic oscillations in nonlinear Schr\"{o}dinger
equations. II. The $L^2$-critical case}, Trans. Amer. Math. Soc.
\textbf{359} (1) (2007), 33-62



\bibitem{chle} M. Chae, S. Hong and S. Lee, {\it Mass concentration
for the $L^2$-critical nonlinear Schr\"{o}dinger equations of higher
orders}, Discrete Contin. Dyn. Syst. \textbf{29} (2011), no. 3,
909-928.

\bibitem{chho} Y. Cho, H. Hajaiej, G. Hwang, and T. Ozawa, {\it On the Cauchy problem of fractional Schr\"{o}dinger equation with Hartree type nonlinearity}, to appear in Funkcialaj Ekvacioj (arXiv:1209.5899).

\bibitem{chkl} Y. Cho, G. Hwang, S. Kwon and S. Lee, {\it On the finite time blowup for Hartree equations}, in preprint.

\bibitem{chkl2} \bysame, {\it Profile decompositions and blowup phenomena of mass critical fractional Schr\"odinger equations}, Nolinear Analysis \textbf{86} (2013), 12-29.

\bibitem{chl} Y. Cho, G. Hwang and S. Lee, {\it An endpoint Strichartz estimate in polar coordinates}, RIMS Kokyuroku Bessatsu B33 (2012), 49-57

\bibitem{cholee} Y. Cho and S. Lee, {\it Strichartz estimates in spherical
coordinates}, Indiana Univ. Math. J., to appear. (arXiv:1202.3543v2)

\bibitem{chonak}  Y. Cho and K. Nakanishi, {\it On the global existence of semirelativistic Hartree equations}, RIMS Kokyuroku Bessatsu, \textbf{B22} (2010), 145-166.

\bibitem{favi}  L. Fanelli and N. Visciglia, \emph{The lack of compactness in the
Sobolev-Strichartz inequalities}, J. Math. Pures Appl. \textbf{99} (2013), 309-320.

\bibitem{frohlenz2} J. Fr\"{o}hlich and E. Lenzmann, {\it Blow-up for nonlinear wave equations describing Boson
stars}, Comm. Pure Appl. Math., \textbf{60} (2007), 1691-1705.



\bibitem{glnw} Z. Guo, S. Lee, K. Nakanishi and C. Wang, {\it Generalized Strichartz estimates and scattering for 3D Zakharov system}, to appear Comm. Math. Phy.

\bibitem{guwa} Z. Guo and Y. Wang, {\it Improved Strichartz estimates for a class of dispersive equations in the radial case and their applications to nonlinear Schr\"{o}dinger and wave
equations}, in preprint. (arXiv:1007.4299v3)



\bibitem{iopu}  A. D. Ionescu and F. Pusateri, {\it Nonlinear fractional Schr\"odinger equations in one dimension}, J. Funct. Anal \textbf{266} (2014), 139-176.



\bibitem{ke} Y. Ke, {\it Remark on the Strichartz estimates in the radial case} J. Math. Anal. Appl. \textbf{387} (2012), no. 2, 857-861.

\bibitem{keme} C. Kenig and F. Merle, {\it Global well-posedness, scattering and blow up for the
energy-critical, focusing, non-linear Schr\"{o}dinger equation in the
radial case}, Invent. Math., \textbf{166} (3) (2006), 645-675.


\bibitem{ker2} S. Keraani, {\it On the blow up phenomenon of the critical nonlinear Schr\"odinger equation}, J.
Funct. Anal. \textbf{235} (2006), 171-192.

\bibitem{ksv} R. Killip, B. Stovall and M. Visan, {\it Scattering for the cubic
Klein-Gordon equation in two space dimensions}, Trans. Amer. Math.
Soc.,  to appear.



\bibitem{la1}  N, Laskin, {\it Fractional quantum mechanics and L\'{e}vy path
integrals}, Phys. Lett. A, \textbf{268} (2000), 298-305.




\bibitem{meve} F. Merle and L. Vega, {\it Compactness at blow-up time for L2
solutions of the critical nonlinear Schr\"odinger equation in 2D},
Internat. Math. Res. Notices 1998, no. 8, 399-425.

\bibitem{mvv} A. Moyua, A. Vargas and L. Vega, \emph{Schr\"odinger maximal function
and restriction properties of the Fourier transform}, International
Math. Research Notices 1996, no. 16, 793-815.


\bibitem{ram} J. Ramos, \emph{A refinement of the Strichartz inequality for the wave equation with applications},
Adv. Math. \textbf{230} (2) (2012), 649-698.

\bibitem{sh} S. Shao, {\it The linear profile decomposition for the Airy equation and the
existence of maximizers for the Airy Strichartz inequality}, Anal.
PDE, \textbf{2} (1) (2009), 83-117.

\bibitem{sw} E. M. Stein and G. Weiss, {\it Introduction to Fourier Analysis on Euclidean
Spaces}, Princeton Univ. Press, 1971.

\bibitem{wat} G. Watson, {\it A Treatise on the Theory of Bessel
Functions}, Reprint of the second (1944) edition. Cambridge
Mathematical Library. Cambridge University Press, Cambridge, 1995.
\end{thebibliography}
\end{document}